\documentclass[a4paper,11pt,reqno]{amsart}
\usepackage[dvipdfm]{hyperref}
\title[free boundary problem for the Fisher-KPP equation]{A free boundary problem for the Fisher-KPP equation with a given moving boundary}
\author[
H. Matsuzawa]{
Hiroshi Matsuzawa$^\ddag$}
\thanks{2010 Mathematics Subject Classification. 35R35, 35K20, 35K55, 35K57, 35K58}
\thanks{{\it Key words and phrases.} free boundary problem, logistic equation, Fisher-KPP equation}
\thanks{$^\ddag$ National Institute of Technology, Numazu College, 3600 Ooka, Numazu City, Shizuoka 410-8501, Japan. (Email: hmatsu@numazu-ct.ac.jp)}
\thanks{The author was partly supported by JSPS KAKENHI Grant-in-Aid for Scientific Research (C) 17K05340.}
\date{\today}
\usepackage{amsmath, enumerate, color, mathrsfs, cases}

\makeatletter
\renewcommand{\theequation}{%
\thesection.\arabic{equation}}
\@addtoreset{equation}{section}
\makeatother
\theoremstyle{definition}

\theoremstyle{plain}

\newtheorem{theorem}{Theorem}
\newtheorem{proposition}{Proposition}[section]
\newtheorem{thm}[proposition]{Theorem}
\newtheorem{corollary}[proposition]{Corollary}
\newtheorem{lemma}[proposition]{Lemma}

\theoremstyle{definition}
\newtheorem{remark}[proposition]{Remark}

\begin{document}
\begin{abstract}
We study free boundary problem of Fisher-KPP equation $u_t=u_{xx}+u(1-u),\ t>0,\ ct<x<h(t)$.
The number $c>0$ is a given constant, $h(t)$ is a free boundary which is determined by the Stefan-like 
condition. This model may be used to describe the spreading of a non-native species over a one dimensional habitat. The free boundary $x=h(t)$ represents the spreading front. 
In this model, we impose zero Dirichlet condition at left moving boundary $x=ct$. This means that 
the left boundary of the habitat is a very hostile environment and that the habitat is eroded away by the left moving boundary at constant speed $c$.  

In this paper we will give a trichotomy result, that is, for any initial data, exactly one of the three behaviours, vanishing, spreading and transition, happens. This result is related to the results 
appears in the free boundary problem for the Fisher-KPP equation with a shifting-environment, which was considered by Du, Wei and Zhou \cite{DWZ}. However the vanishing in our problem is different from that in \cite{DWZ}
 because in our vanishing case, the solution is not global-in-time.
\end{abstract}
\maketitle
\section{Introduction and Main Results}
We consider the following free boundary problem for the Fisher-KPP equation:
\begin{equation}\label{fbp}
\begin{cases}
u_t=u_{xx}+u(1-u),          &t>0,\ ct<x<h(t),\\
u(t,ct)=u(t,h(t))=0,      &t>0, \\
h'(t)=-\mu u_x(t,h(t)),   &t>0, \\
h(0)=h_0,\ u(0,x)=u_0(x), &0\leq x\leq h_0,
\end{cases}
\end{equation}
where $c$, $\mu$ and  $h_0$ are given positive constants, so $x=ct$ is a given forced moving boundary with speed $c$. The right moving boundary 
$x=h(t)$ is to be determined together with $u(t,x)$.
Initial function $u_0$ belongs to $\mathscr{X}(h_0)$ for some $h_0>0$, where
\begin{align*}
\mathscr{X}(h_0):=\left\{\phi\in C^2[0, h_0]:\begin{array}{l}\phi(0)=\phi(h_0)=0,\ \\ \phi'(0)>0,\ \phi'(h_0)<0,\ \phi(x)>0\ {\rm in}\ (0,h_0) \end{array}\right\}.
\end{align*}
For any $h_0>0$ and $u_0\in\mathscr{X}(h_0)$, we say a pair $(u(t,x), h(t))$ a classical solution of \eqref{fbp} on time interval $[0, T]$ for some $T>0$ 
if it satisfies $u\in C^{1,2}(G_T)$ and  $h\in C^1([0,T])$ and all the identities in \eqref{fbp} are satisfied pointwisely where 
\begin{align*}
G_T:=\{(t,x) : t\in(0,T],\ x\in [ct, h(t)]\}.
\end{align*}

This model may be used to describe the spreading of a new or invasive species with population density $u(t,x)$ over a one dimensional habitat. 
The free boundary $x=h(t)$ represents the spreading front. The behavior of the free boundary is determined by the Stefan-like condition which implies 
that the population pressure at the free boundary is driving force of the spreading front. In this model, we impose zero Dirichlet boundary condition at left moving boundary $x=ct$. This means that 
the left boundary of the habitat is a very hostile environment for the species and that the habitat is eroded away by the left moving boundary at constant speed $c$. 

Recently, problem (\ref{fbp}) with $c=0$ was studied in pioneering paper \cite{DLI}(in which Neumann boundary condition is imposed at left fixed boundary $x=0$), \cite{KOY} and \cite{KY}. 
The authors showed  that (\ref{fbp}) has a unique solution which is defined for all $t>0$ and, as $t\to\infty$, the interval $[0, h(t)]$ converges to either a finite interval $[0, h_{\infty})$ or 
$[0, \infty)$. Moreover, in the former case, $u(t,x)\to 0$ uniformly in $x$, while in the latter case, $u(t,x)\to 1$ locally uniformly in $[0, \infty)$. See also \cite{DL} for the double 
fronts free boundary problem with monostable, bistable or combustion type nonlinearity. Moreover, in the case of spreading, it is shown in \cite{DLI, DL} that there exists $c^*=c^*(\mu)>0$ such that 
$\lim_{t\to\infty}(h(t)/t)=c^*$. In this sense, $c^*$ is called the asymptotic spreading speed of corresponding free boundary problems. In \cite{DL}, the authors showed that $c^*$ is determined 
by the unique solution pair $(c,q)=(c^*, q^*)$ of the following problem
\begin{align*}
\left\{
\begin{array}{l}
q''+cq+q(1-q)=0,\ \ z\in (-\infty, 0), \\
q(0)=0,\ q(-\infty)=1,\ q'(0)=-c/\mu,\ q(z)>0\ \ z\in(-\infty, 0).
\end{array}\right.
\end{align*}

Using a simple variation of the techniques in \cite{DLI}, we can see that for any $h_0>0$ and $u_0\in\mathscr{X}(h_0)$, \eqref{fbp}(or \eqref{fbp2} with quite general nonlinearity $f$) has a 
unique solution defined on some time interval $[0, T]$ and it can be extended to some wider time interval $[0, \overline{T}]$ with $\overline{T}>T$ whenever $\inf_{t\in[0,T]}(h(t)-ct)>0$ is satisfied
(see Proposition \ref{existence} and Lemma \ref{hprime}). Therefore, for any $h_0>0$ and $u_0\in\mathscr{X}(h_0)$ we can define the maximal existence time $T^*$ of solution to \eqref{fbp} in the following way:
\begin{align}\label{Tmax}
T^*:=\sup\{T>0:(u,h)\ {\rm is\ the\ solution\ to}\ \eqref{fbp}\ {\rm on}\ [0,T]\}.
\end{align}
We say $(u,h)$ is a classical solution of \eqref{fbp} on time interval $[0,T^*)$ if for any $T\in (0, T^*)$, $(u,h)$ is a classical solution of \eqref{fbp} on time interval $[0,T]$.

The main purpose of this paper is to study the  behavior of solutions to (\ref{fbp}). When $T^*=\infty$, the solution is 
global and so we can study its asymptotic behavior. On the other hand, in this problem, $T^*$ may be a finite number for the reason that $h(t)-ct\to 0$ as $t\nearrow T^*$, that is the habitat of the species may shrink to a 
single point. Such a phenomenon is observed first in free boundary problems considered by \cite{C, CLZ}. We concern with the following questions:
\begin{enumerate}[(1)]
\item When the situation that $T^*<\infty$ and $h(t)-ct\to 0$ as $t\nearrow T^*$ occur? 
\item Can the situation that $T^*=\infty$ and $h(t)-ct\to 0$ as $t\to \infty$ occur?
\item When $T^*<\infty$ and $h(t)-ct\to 0$ as $t\nearrow T^*$, how about the behavior of $u$ as $t\nearrow T^*$ is ? 
\item When $T^*=\infty$, reveal all possible long-time dynamical behavior of the solutions.
\end{enumerate}

Now we state our main theorems. First theorem is a trichotomy result for the case $0<c<c^*$.

\begin{theorem} Suppose that $0<c<c^*$ and $(u,h)$ is the unique solution of \eqref{fbp} on a time interval $[0, T^*)$ with $T^*$ maximal existence time. Then exactly one of the following happens:
\begin{enumerate}[{\rm (1)}]
\item {\bf Vanishing:} $T^*<\infty$, $\lim_{t\nearrow T^*}(h(t)-ct)=0$, 
\begin{align*}
\lim_{t\nearrow T^*}\left\{\max_{x\in [ct, h(t)]}u(t,x)\right\}=0.
\end{align*}
\item {\bf Spreading:} $T^*=\infty$, $\lim_{t\to\infty}(h(t)/t)=c^*$ and for any small $\varepsilon>0$
\begin{align*}
\lim_{t\to\infty}\left\{\max_{x\in[(c+\varepsilon)t, (c^*-\varepsilon)t]}|u(t,x)-1|\right\}=0.
\end{align*}
\item {\bf Transition:} $T^*=\infty$, $\lim_{t\to\infty}(h(t)-ct)=L_c$ and
\begin{align*}
\lim_{t\to\infty}\left\{\max_{x\in[ct,h(t)]}|u(t,x)-V_c(x-h(t)+L_c)|\right\}=0,
\end{align*}
where $L_c>0$ are determined by a unique solution pair $(L, V)=(L_c, V_c)$ to the problem
\begin{align*}
\left\{ 
\begin{array}{l}
V''+cV'+V(1-V)=0,\ V>0\ \ {\rm for}\ \ z\in (0, L), \\
V(0)=V(L)=0,\ -\mu V'(L)=c.
\end{array}\right.
\end{align*}
\end{enumerate}
\end{theorem}

If the initial function $u_0$ in \eqref{fbp} has the form $u_0=\sigma\phi$ with some fixed $\phi\in\mathscr{X}(h_0)$, we can obtain the following sharp threshold result.
\begin{theorem}
Suppose that the initial function $u_0$ in \eqref{fbp} has the form $u_0=\sigma\phi$ with some fixed $\phi\in\mathscr{X}(h_0)$. Then 
there exists $\overline{\sigma}\in(0,\infty]$ such that vanishing happens when $0<\sigma<\overline{\sigma}$, spreading happens when $\sigma>\overline{\sigma}$, and transition happens when 
$\sigma=\overline{\sigma}$.
\end{theorem}

When $c\ge c^*$, vanishing always happens.

\begin{theorem} Assume that $c^*\le c$ and $(u,h)$ is the unique solution of \eqref{fbp} on a time interval $(0, T^*)$ with $T^*$ maximal existence time. Then we have $T^*<\infty$ and 
$\lim_{t\nearrow T^*}(h(t)-ct)=0$ and $\lim_{t\nearrow T^*}\sup_{x\in [ct, h(t)]}u(t,x)=0$.
\end{theorem}

The trichotomy result of Theorem A is related to the result of \cite{DWZ}, where a free boundary problem of Fisher-KPP equation with shifting-environment is considered. The sifting-environment there is given in the 
nonlinearity with the form $A(x-ct)u-bu^2$, 
where $A(\xi)$ is a Lipschitz continuous function on $\mathbb{R}^1$ which satisfies
\begin{align*}
A(\xi)=\left\{
\begin{array}{ll}
a_0, & \xi<-l_0, \\
a, & \xi\ge 0,
\end{array}\right.
\end{align*}
and $A(\xi)$ is strictly increasing on $[-l_0, 0]$. Here $l_0$, $a_0$ and $a$ are constants, with $l_0\ge 0$, $a_0\le 0$ and $a>0$. In the model, set 
$\{x\in\mathbb{R}^1: x-ct\le -l_0\}$ represents the unfavourable range of the environment and the range move with constant speed $c>0$, which corresponds to the very hostile boundary $x=ct$ 
of our model. However, comparing with the results in \cite{DWZ}, the solutions to our problem become non-global in the vanishing case. 
This is  significantly different from the model of \cite{DWZ}. As far as I know, there are relatively few free boundary problem of this kind which have 
non-global solutions (see \cite{C,CLZ}). The appearance of non-global solutions can make our model more realistic because some species become extinct in finite time due to shrinking  of 
their habitats.       

Furthermore, from a mathematical point of view, our main results can be seen as a drastic change of classification of behaviors of solutions, which is caused only by replacing of 
left fixed boundary $x=0$ by moving boundary $x=ct$, but remaining the nonlinearity unchanged, in the problems considered earlier in \cite{DLI, KY, KOY}. 

Because, in the present paper, some approaches rely on the special form of the logistic nonlinearity, 
it should be more challenging to consider the problem \eqref{fbp} with logistic nonlinearity $u(1-u)$
 replaced by general monostable, bistable or combustion type nonlinearity. This will be considered in forthcoming paper \cite{KM2}.
 
The rest of this paper is organized as follows. In section 2, we will present some basic results. Section 3 will deal with the situation $T^*<\infty$. Section 4 will be devoted to the proof of Theorem A. 
In section 5, we will prove Theorem B.  

\section{Preliminary results}

In this section we give some preliminary results. The results here except Proposition \ref{compact-support-wave-ex}, Lemma \ref{auxelliptic} and 
Proposition \ref{upperestspeed} valid for rather general nonlinearity. In this section, we assume that
\begin{align}\label{assumptionf}
f\in C^1,\ \ f(0)=f(1)=0,\ \ f'(1)<0,\ \ f(u)<0\ \ {\rm for}\ \ u>1
\end{align}
and consider
\begin{equation}\label{fbp2}
\begin{cases}
u_t=u_{xx}+f(u),          &t>0,\ ct<x<h(t),\\
u(t,ct)=u(t,h(t))=0,      &t>0, \\
h'(t)=-\mu u_x(t,h(t)),   &t>0, \\
h(0)=h_0,\ u(0,x)=u_0(x), &0\leq x\leq h_0,
\end{cases}
\end{equation}
instead of \eqref{fbp}. 

\subsection{Existence of the local solution} 
The local existence and uniqueness result can be proved by using contraction mapping principle as in \cite{DLI}.
\begin{proposition}\label{existence}
For any $h_0>0$, $u_0\in\mathscr{X}(h_0)$ and $\alpha\in(0,1)$, there exists $T>0$ such that problem \eqref{fbp2} admits a unique solution $(u, h)$ defined on $[0, T]$ with 
\begin{align*}
u\in C^{\frac{1+\alpha}{2}, 1+\alpha}(\overline{G_T}),\ h\in C^{1+\frac{\alpha}{2}}([0, T]),
\end{align*}
where $G_T:=\{(t,x)\in\mathbb{R}^2 : t\in (0, T], x\in [ct, h(t)]\}$. Moreover we have
\begin{align*}
\|u\|_{C^{\frac{1+\alpha}{2}, 1+\alpha}(G_T)}+\|h\|_{C^{1+\frac{\alpha}{2}}([0,T])}\le C,
\end{align*}
where $C$ and $T$ depend only on $h_0$, $\alpha$ and $\|u_0\|_{C^2[0,h_0]}$.
\end{proposition}

\begin{remark}
As in \cite{DLI}, by applying the Schauder estimate to the equivalent fixed boundary value problem used in the proof, we can derive an additional regularity for $u$, 
namely $u\in C^{1+\frac{\alpha}{2}, 2+\alpha}(D_T)$.
\end{remark}

Next two lemmas are about a priori estimates for $u$ and $h'$.

\begin{lemma}\label{ODE-estimate}
Suppose that $(u,h)$ be a global solution to \eqref{fbp2}. Then for any $\delta\in (0, -f'(1))$ there exists $M>0$ such that $u(t,x)\le 1+Me^{-\delta t}$ for $t>0$ and $x\in [ct, h(t)]$.
\end{lemma}
\begin{proof}
We first note that by the condition on $f$, for any $\delta\in(0, -f'(1))$, there exists $\rho=\rho(\delta)>0$ such that 
\begin{align}\label{flinear}
f(u)\ge\delta (1-u)\ \ (u\in[1-\rho, 1]),\ \ \ f(u)\le\delta(1-u)\ \ (u\in[1, 1+\rho]).
\end{align}
Consider the solution to the following initial value problem of ordinary differential equation:
\begin{align*}
\left\{
\begin{array}{l}
\displaystyle\frac{d\overline{u}}{dt}=f(\overline{u}), \\ 
\overline{u}(0)=C_1:=\max\{1, \|u_{0}\|_{C[0,h_{0}]}\}.
\end{array}
\right.
\end{align*}
Then the standard comparison principle shows that
\begin{align}\label{comparewithODEsol}
u(t,x)\le\overline{u}(t)\ \ {\rm for}\ \ t>0,\ ct<x<h(t).
\end{align}
We note that $\overline{u}(t)$ is monotone decreasing and it converges to $1$ as $t\to\infty$. Hence there exists $T>0$ such that $\overline{u}(t)\le 1+\rho$ for $t\ge T$. It follows from \eqref{flinear} that $\overline{u}=\overline{u}(t)$ 
satisfies
\begin{align*}
\left\{
\begin{array}{l}
\displaystyle\frac{d\overline{u}}{dt}=f(\overline{u})\le\delta(1-\overline{u}), t>T, \\ 
\overline{u}(T)\le 1+\rho.
\end{array}
\right.
\end{align*}
and then $\overline{u}(t)\le 1+Me^{-\delta t}$, where $M=\rho e^{\delta T}$. From \eqref{comparewithODEsol}, we obtain the desired inequality.
\end{proof}

\begin{lemma}\label{hprime}
Let $(u,h)$ be any solution of \eqref{fbp2} for $0<t<T_0$ with some $T_0\in (0,\infty)$. Then the solution satisfies
\begin{align*}
&0<u(t,x)\le C_1\quad \mbox{for}\ \ 0<t<T_0,\ ct<x<h(t),\\
&0<h'(t)\le \mu C_2\quad \mbox{for}\ \ 0<t<T_0,
\end{align*}
where $C_1$ and $C_2$ are positive constants independent of $T_0$. 

Moreover the solution can be extended to some interval $[0, \overline{T}]$ with $\overline{T}>T_0$ if $\inf_{t\in (0,T_0)}[h(t)-ct]>0$.
\end{lemma}

\begin{proof}
By the strong maximum principle we have
\begin{align}\label{smp}
\begin{split}
u(t,x)>0\ \ \mbox{for}\ \ 0<t<T_0,\ ct<x<h(t),\\
u_x(t,h(t))<0\ \ \mbox{for}\ \ 0<t\le T_0.
\end{split}
\end{align}
Let $C_1:=\max\{1, \|u_0\|_{C[0,h_0]}\}$. By the proof of Lemma \ref{ODE-estimate}, we can obtain
\begin{align*}
u(t,x)\leq \overline{u}(t)\le C_1\ \ \mbox{for}\ \ 0<t<T_0,\ ct<x<h(t).
\end{align*}
We will next prove $0<h'(t)\le \mu C_2$ for some $C_2>0$.
From \eqref{smp} we see $h'(t)=-\mu u_x(t,h(t))>0$ for $0<t<T_0$,
and it remains to prove $h'(t)\le \mu C_2$.
Define
\begin{align*}
w(t,x)=-C_1M^2(x-h(t))(x-h(t)+2/M)
\end{align*}
and we use a comparison principle over
\begin{align}\label{DeltaM}
\Delta_M=\{(t,x)\in \mathbb{R}^2;\ 0<t<T_0,\ \max\{ct, h(t)-1/M\}<x<h(t)\}.
\end{align}
Here we choose large $M$ satisfying
\begin{align}\label{Mdef}
M=\max\left\{\frac{\sqrt{2K}}{2},\ \frac{\|u_0'\|_{C([-h_0,h_0])}}{C_1}\right\}
\end{align}
with $K=\displaystyle \max_{0\le w\le C_1}|f'(w)|$.
Direct calculation gives
\begin{align*}
&w_t=2C_1M^2h'(t)(x-h(t)+1/M)\ge 0\quad \mbox{in}\ \ \Delta_M,\\
&w_x=-2C_1M^2(x-h(t)+1/M),\\
&w_{xx}=-2C_1M^2.
\end{align*}
Using \eqref{Mdef}, we have
\begin{align*}
w_t-w_{xx}-f(w)
&\ge 2C_1M^2-C_{1}K\\
&\ge C_1(2M^2-K)\\
&\ge 0\quad \mbox{in}\ \ \Delta_M.
\end{align*}
We next note that
\begin{align*}
&w(t,h(t))=u(t,h(t))=0, \\
&w(t,h(t)-1/M)=C_1\ge u(t,h(t)-1/M)\ \ {\rm when}\ \ ct<h(t)-1/M, \\
&w(t,ct)>0=u(t,ct)\ \ {\rm when}\ \ h(t)-1/M\le ct 
\end{align*}
for $0<t\le T$. Note that
\begin{align*}
u_0(x)&= \int_{h_0}^xu_0'(y)\ dy\le \|u_0'\|_{C([-h_0,h_0])}(h_0-x),\\
w(0,x)&=C_1M^2(h_0-x)(x-h_0+2/M)\ge C_1M(h_0-x)
\end{align*}
for $x\in[0,h_0]\cap [h_0-1/M, h_0]$. By \eqref{Mdef} we obtain
\begin{align*}
u_0(x)\le \|u_0'\|_{C([-h_0,h_0])}(h_0-x)\le C_1M(h_0-x)\le w(0,x)
\end{align*}
for $x\in[0,h_0]\cap[h_0-1/M,h_0]$.
Hence the standard comparison principle implies
\begin{align*}
u(t,x)\le w(t,x)\ \ \mbox{in}\ \ \Delta_M.
\end{align*}
Since $u(t,h(t))=w(t,h(t))=0$ for $0<t<T_0$, we have $u_x(t,h(t))\ge w_x(t,h(t))$ for $0<t<T_0$. Therefore
\begin{align*}
h'(t)=-\mu u_x(t,h(t))\le -\mu w_x(t,h(t))=\mu (2C_1M)=:\mu C_2
\end{align*}
for $0<t<T_0$.

Now we assume $\rho:=\inf_{t\in(0,T_0)}[h(t)-ct]>0$ and prove that the solution $(u, h)$  can be extended to some interval $[0,\overline{T}]$ with $\overline{T}>T_0$. From above estimates we have
\begin{align*}
h(t)\in [h_0, h_0+\mu C_2t], h'(t)\in (0, \mu C_2]\ \ {\rm for}\ t\in (0,T_0).
\end{align*}
We now fix $\delta\in(0, T_0)$ By standard $L^p$ estimates, the Sobolev embedding theorem, and the Schauder estimates for parabolic equations, we can find $C_3>0$ depending only on 
$\delta$, $T_0$, $C_1$, $C_2$ such that $\|u(t,\cdot)\|_{C^2[ct, h(t)]}\le C_3$ for $t\in[\delta, T_0)$. It then follows from the proof of Theorem \ref{existence} (cf \cite{DLI}) that 
there exists $\tau>0$ depending only on $C_1$, $C_2$ and $C_3$ and $\rho$ but not on $t$ such that the solution to problem (\ref{fbp2}) with initial time $t\in[\delta, T_0)$ 
can be extended uniquely to the time $t+2\tau$. In particular, if we start from 
 time $T_0-\tau$, then we can extend to the solution to time $T_0+\tau$.
\end{proof}

Now, for any $h_0>0$ and $u_0\in\mathscr{X}(h_0)$, we can define the maximal existence $T^*\in (0, \infty]$ of solution $(u,h)$ to \eqref{fbp} as in \eqref{Tmax}.

\subsection{Comparison principles}
In the proof the main theorems, we will frequently construct suitable upper and lower solutions. 
\begin{lemma}\label{comp}
Let $\xi$, $\overline{h}\in C^1([0,T])$ and $\overline{u}\in C(\overline{D}_T)\cap C^{1,2}(D_T)$ with 
$D_T=\{(t,x)\in\mathbb{R}^2 : 0<t\le T,\ \xi(t)<x<\overline{h}(t)\}$ for $T\in (0,\infty)$ satisfy
\begin{align*}
\left\{
\begin{array}{ll}
\overline{u}_t-\overline{u}_{xx}-f(\overline{u})\ge 0,& 0<t\le T,\ \xi(t)<x<\overline{h}(t), \\ 
\overline{u}(t,\overline{h}(t))=0,& 0<t\le T,\\ 
\overline{h}^{\prime}(t)\ge -\mu\overline{u}_x(t,\overline{h}(t)),& 0<t\le T.\\ \end{array}\right.
\end{align*}
For a solution $(u, h)$ to \eqref{fbp2}, if 
\begin{align*}
&ct\le \xi(t),\ u(t, \xi(t))\le \overline{u}(t, \xi(t))\ \ \mbox{for}\ \ 0<t\le T,\\
&h_0\le \overline{h}(0),\ u_0(x)\le\overline{u}(0,x)\ \ {\it for}\ \ \xi(0)\le x\le h_0,
\end{align*}
then
\begin{align*}
&h(t)\le \overline{h}(t)\ \ {\it for}\ 0<t\le T,\\
&u(t,x)\le\overline{u}(t,x)\ \ {\it for}\ 0<t\le T,\ \xi(t)<x<h(t).
\end{align*}
\end{lemma}

\noindent The function $\overline{u}$ or the pair $(\overline{u}, \overline{h})$ in Lemma \ref{comp} is usually called an 
upper solution of problem \eqref{fbp2}. We can define a lower solution by reversing all the inequalities in suitable places. There is a symmetry version of Lemma 
\ref{comp}, where the conditions on the left and right boundaries are interchanged. We also have corresponding comparison results for lower solutions in each case.

\subsection{Zero number arguments}
Our arguments in the present paper rely on the zero number argument that depends on the result of Angenent \cite{A}. For later use, we give a basic result of the zero number argument, which is a variant of Theorem C and D in \cite{A}. See also \cite{DLZ}.
\begin{lemma}\label{zeronumber0} 
$u:[0,T]\times[0, 1]\to\mathbb{R}$ be a bounded classical solution of 
\begin{align}\label{zeronumbereq}
u_{t}=a(t,x)u_{xx}+b(t,x)u_{x}+c(t,x)u
\end{align}
with boundary conditions
\begin{align*}
u(t,0)=l_{0}(t),\ u(t, 1)=l_{1}(t),
\end{align*}
where $l_{0}$, $l_{1}\in C^{1}[0,T]$, and $l_{0}$ and $l_{1}$ satisfies
\begin{align*}
l_{i}(t)\equiv 0 \ \ {\rm for}\ \ t\in[0,T]\ \ {\rm or}\ \ l_{i}(t)\ne 0\ \ {\rm for\ any}\ \ t\in[0,T]
\end{align*}
for each $i=0,1$. Assume that 
\begin{align*}
a, 1/a, a_t, a_x, a_{xx}, b, b_t, b_x, c\in L^{\infty},\ a>0\ \ {\it and}\ \ u(0,\cdot)\not\equiv 0\ \ {\it when}\ \ l_0=l_1=0.
\end{align*}Let $z(t)$ denote the number of zeros of $u(t,\cdot)$ in $[0,1]$. Then
\begin{enumerate}[{\rm (a)}]
\item for each $t\in(0,T]$, $z(t)$ is finite,
\item $z(t)$ is nonincreasing in $t$,
\item if for some $s\in (0, T)$ the function $u(s, \cdot)$ has a degenerate zero $x_{0}\in [0, 1]$, that is, 
\begin{align*}
u(s, x_{0})=u_{x}(s, x_{0})=0
\end{align*}
holds, then $z(t_{1})>z(t_{2})$ for all $t_{1}<s<t_{2}$.
\end{enumerate}
\end{lemma}

\begin{lemma}\label{zeronumber}
Let $\xi_1(t)$ and $\xi_2(t)$ be continuous functions of $t\in (t_0, t_1)$ and assume that $\xi_1(t)<\xi_2(t)$ for $t\in(t_0, t_1)$. 
Suppose $u(t,x)$ is a continuous function of $t\in(t_0, t_1)$ and 
$x\in[\xi_1(t), \xi_2(t)]$, and satisfies \eqref{zeronumbereq} in the classical sense for $t\in(t_0, t_1)$ and $x\in (\xi_1(t), \xi_2(t))$ with
\begin{align*}
u(t, \xi_1(t))\ne 0,\ u(t, \xi_2(t))\ne 0\ \ {\it for}\ t\in(t_0, t_1).
\end{align*}
Let $Z(t)$ denote the number of zeros of $u(t, \cdot)$ in $[\xi_1(t), \xi_2(t)]$. Then
\begin{enumerate}[{\rm (a)}]
\item for each $t\in (t_0, t_1)$, $Z(t)$ is finite,
\item $Z(t)$ is nonincreasing in $t$,
\item if for some $s\in(t_0, t_1)$ the function $u(s, \cdot)$ has a degenerate zero $x_0\in(\xi_1(s), \xi_2(s))$, that is, 
\begin{align*}
u(s,x_0)=u_x(s, x_0)=0,
\end{align*}
holds, then $Z(s_1)>Z(s_2)$ for all 
$s_1$, $s_2$ satisfying $t_0<s_1<s<s_2<t_1$.
\end{enumerate}
\end{lemma}
You can find the proof of Lemma \ref{zeronumber} in \cite{CLZ} and \cite{DLZ}.

\subsection{Traveling waves and an auxiliary problem}

First we consider the following problem
\begin{align}\label{semi-wave-problem}
\left\{
\begin{array}{l}
q''+cq'+q(1-q)=0\ {\rm in}\ (-\infty, 0), \\
q(0)=0,\ q(-\infty)=1,\ q(z)>0\ {\rm in}\ (-\infty, 0).
\end{array}\right.
\end{align}

\begin{proposition}[Proposition 1.8 and Theorem 6.2 of \cite{DL}]\label{semi-wave} For any $\mu>0$ there exists a unique $c^{*}=c^{*}_{\mu}>0$ and a solution $q^{*}$ to 
\eqref{semi-wave-problem} with $c=c^{*}$ such that $(q^{*})'(0)=-c^{*}/\mu$.
\end{proposition}

We remark that this function $q^{*}$ is shown in \cite{DL} to satisfy $(q^{*})'(z)<0$ for $z\le 0$.

We call $q^{*}$ a {\it semi-wave\ with\ speed} $c^{*}$, since the function $w(t,x):=q^{*}(x-c^{*}t)$ satisfies
\begin{align*}
\left\{
\begin{array}{l}
w_{t}=w_{xx}+w(1-w)\ \ {\rm for}\ \ t\in\mathbb{R}^{1},\ x<c^{*}t, \\
w(t, c^{*}t)=0,\ w_{x}(t, c^{*}t)=-c^{*}/\mu,\ w(t,-\infty)=1,\ \ t\in\mathbb{R}^{1}.
\end{array}
\right.
\end{align*}
\begin{remark} We remark that the number $c^*=c^*(\mu)$ satisfies $c^*\in (0, c_0)$ and $\lim_{\mu\to\infty}c^*(\mu)=c_0$, where $c_0$ is the minimal speed of traveling wave (see \cite{DL}). 
In case the nonlinearity is $u(1-u)$, $c_0=2$ holds. This proposition holds for monostable, bistable and combustion type nonlinearities. For these types of nonlinearities, the number 
$c^*$ satisfies $c^*\in(0,c_0)$ and $\lim_{\mu\to\infty}c^*(\mu)=c_0$, where $c_0$ express the minimal speed of traveling wave when the nonlinearity is of monostable, the unique speed of traveling wave when
 the nonlinearity is of bistable or combustion type(see \cite{AW1} and \cite{AW2}). 
\end{remark}

Next we consider the following problem :
\begin{align}\label{compact-support-wave}
\left\{
\begin{array}{l}
V''+cV'+V(1-V)=0,\ \ 0<z<L. \\
V(0)=0,\ V(L)=0,\ V(z)>0\ {\rm in}\ (0,L).
\end{array}
\right.
\end{align}

By virtue of a phase-plane analysis in case (iv) of Section 3.2 in \cite{GLZ} (see also \cite{DWZ}), we have the following proposition.

\begin{proposition}[\cite{GLZ}]\label{compact-support-wave-ex}
For any $c\in(0,c^{*})$, there exist a unique positive number $L_{c}$ and a  unique solution $V_{c}$ to \eqref{compact-support-wave} with $L$ replaced $L_{c}$ such that $V'_{c}(L_{c})=-c/\mu$.
\end{proposition}

If we define the function $w(t,x):=V_{c}(x-ct)$, $w$ satisfies
\begin{align*}
\left\{
\begin{array}{l}
w_{t}=w_{xx}+w(1-w)\ \ {\rm for}\ \ t\in\mathbb{R}^{1},\ ct<x<ct+L_{c}, \\
w(t, ct)=w(t, ct+L_{c})=0,\ \ t\in\mathbb{R}^1, \\
-\mu w_{x}(t, ct+L_{c})=c,\ t\in\mathbb{R}^{1}.
\end{array}
\right.
\end{align*}
and $w$ resemble a traveling wave with a compact support moving to the right at constant speed $c$. 

We next state the following lemma on an auxiliary elliptic problem for later use.
\begin{lemma}[Lemma 2.4 of \cite{DWZ}]\label{auxelliptic}
Suppose that $C\in [0, c_0)$. Then for all large $l>0$, the problem
\begin{align*}
\left\{
\begin{array}{ll}
w''+Cw'+w(1-w)=0, & {\it for}\ \ -l<x<l,\\ 
w(-l)=w(l)=0, & 
\end{array}\right.
\end{align*}
admits a unique positive solution $w_l(x)$. Moreover, $\lim_{l\to\infty}w_l(x)=1$ uniformly in any compact subset of $\mathbb{R}^1$.
\end{lemma}

\subsection{An upper estimate of $h(t)$}
At the end of this section, we obtain an upper estimate of $h(t)$ for the global solution $(u,h)$. By constructing the upper solution of the form 
\begin{align*}
&\overline{h}(t):=c^*t+M(e^{-\delta T}-e^{-\delta t})+H,
&\overline{u}(t,x):=(1+Me^{-\delta t})q^*(x-\overline{h}(t))
\end{align*}
with suitable $M$, $\delta$, $H$ and $T>0$ as in \cite[Lemma 3.2]{DMZ}, we can obtain the following proposition.
\begin{proposition}\label{upperestspeed}
Assume that $f(u)=u(1-u)$ and let $(u,h)$ be a global solution to \eqref{fbp2}. Then there exists $C_{0}>0$ such that $h(t)-c^{*}t<C_{0}$ for $t>0$.
\end{proposition}
\begin{remark}
This proposition holds for monostable, bistable and combustion type nonlinearities.
\end{remark}

\section{The case of $T^*<\infty$}

In this section, we give some properties of the solutions which exhibit vanishing. The result here also valid for \eqref{fbp2} with $f$ satisfying \eqref{assumptionf}. We assume, in this section, that 
$f$ satisfies \eqref{assumptionf}. The proofs hear are inspired by the methods in \cite{C, CLZ}.
\begin{lemma}\label{fte1}
Let $(u,h)$ be the solution to \eqref{fbp2} on $[0, \hat{T})$. If $\lim_{t\nearrow\hat{T}}[h(t)-ct]=0$, then we have 
$\lim_{t\nearrow\hat{T}}\|u(t,\cdot)\|_{C[ct, h(t)]}=0$.
\end{lemma}
\begin{proof} 
From the proof of Lemma \ref{hprime}, we have
\begin{align*}
u(t,x)\le w(t,x)\ \ {\rm on}\ \ \Delta_M, 
\end{align*}
where 
\begin{align*}
w(t,x)=-C_1M^2(x-h(t))(x-h(t)+2/M).
\end{align*}
and $\Delta_M$ is defined in \eqref{DeltaM}. By our assumption, there exists $\hat{T}_1>0$ such that $h(t)-(1/M)<ct<h(t)$ holds for $t\in(\hat{T}_1, \hat{T})$. Then we have 
\begin{align*}
u(t,x)\le w(t,x)\le 2C_1M(h(t)-ct)\ \ {\rm for}\ t\in(\hat{T}_1,\hat{T})\ x\in [ct, h(t)].
\end{align*}
Letting $t\nearrow \hat{T}_1$, we have $\lim_{t\nearrow\hat{T}}\|u(t,\cdot)\|_{C[ct,h(t)]}=0$. 
\end{proof}

\begin{proposition}\label{fte2}
Let $(u,h)$ be the unique solution to problem \eqref{fbp2} on $[0, \hat{T})$. If $\lim_{t\nearrow\hat{T}}[h(t)-ct]=0$. Then we have $\hat{T}<\infty$.
\end{proposition}
\begin{proof}
For fixed $L>1$, define 
\begin{align*}
\zeta_0(x)=\frac{2\varepsilon}{L^2}(L^2-x^2)
\end{align*}
and consider
\begin{align}\label{zeta}
\begin{cases}
\zeta_t=\zeta_{xx}+\overline{f}(\zeta),       &t>0,\ 0<x<\overline{h}(t),\\
\zeta_x(t,0)=\zeta(t,\overline{h}(t))=0,      &t>0, \\
\overline{h}'(t)=-\mu \zeta_x(t,\overline{h}(t)),   &t>0, \\
\overline{h}(0)=L,\ \zeta(0,x)=\zeta_0(x), &0\leq x\leq L,
\end{cases}
\end{align}
where 
\begin{align*}
\overline{f}(\zeta)=2K\zeta\left(1-\frac{\zeta}{2\varepsilon}\right).
\end{align*}
We note that $f(u)\le\overline{f}(u)$ holds for $u\in[0,\varepsilon]$. By Theorem 2.3 of \cite{DLI}, problem \eqref{zeta} admits a unique global solution $(\zeta, \overline{h})$. Moreover, 
since $|\zeta_{0}|\le 2\varepsilon$ and $\overline{f}(\zeta)<0$ for $\zeta>2\varepsilon$, we have $\zeta(t,x)\le 2\varepsilon$ for $t>0$ and $x\in[0, \overline{h}(t)]$. We again consider following quadratic function $U^{\varepsilon}(t,x)$:
\begin{align*}
U^{\varepsilon}(t,x):=-2\varepsilon M^{2}(x-\overline{h}(t))(x-\overline{h}(t)+2/M)
\end{align*}
over $\hat{\Delta}_{M}:=\{(t,x):t>0, \max\{ct, \overline{h}(t)-1/M\}<x<\overline{h}(t)\}$. Since for $(t,x)\in\hat{\Delta}_{M}$
\begin{align*}
&U^{\varepsilon}_{t}(t,x)=2\varepsilon M^{2}\overline{h}'(t)\left(x-\overline{h}(t)+1/M\right)\ge 0, \\
&U^{\varepsilon}_{xx}(t,x)=-2\varepsilon M^{2},
\end{align*}
we have
\begin{align*}
U^{\varepsilon}_{t}-U^{\varepsilon}_{xx}-\overline{f}(U^{\varepsilon})\ge 2\varepsilon M^{2}-2K\varepsilon\ge 0
\end{align*}
on $\hat{\Delta}_{M}$. By a direct calculation and the assumption $L>1$, we see $\|\zeta_{0}'\|_{C[0, L]}\le 4\varepsilon$. By using comparison principle, 
we obtain $\zeta(t,x)\le U^{\varepsilon}(t,x)$ on $\hat{\Delta}_{M}$. Hence we have
\begin{align*}
-\mu\zeta_{x}(t,\overline{h}(t))\le -\mu U^{\varepsilon}_{x}(t,\overline{h}(t))=4\mu M\varepsilon\left(\le 4\mu\varepsilon\max\left\{\frac{\sqrt{2K}}{2}, 4\right\}\right).
\end{align*}
If we choose $\varepsilon>0$ so small that 
\begin{align}\label{epsilon}
\varepsilon<\frac{c}{4\mu M}
\end{align}
holds, we have $\overline{h}(t)-ct\to 0$ as $t\to\hat{T}_{0}\le\frac{L}{c-4\mu M\varepsilon}$ for some $\hat{T}_0>0$.

Now we fix $\varepsilon>0$ so small that (\ref{epsilon}) holds. By Lemma \ref{fte1}, there exists $\tilde{T}\in(0, \hat{T})$ such that $u(t,x)\le\varepsilon$ holds for $t\in(\tilde{T},\hat{T})$ and $x\in[ct, h(t)]$. If we take $L>\max\{\sqrt{2}h(\tilde{T}), 1\}$, we can easily see that 
$u(\tilde{T},x)\le \zeta_{0}(x)$ for $x\in[c\tilde{T}, h(\tilde{T})]$. By comparison principle we have $h(t+\tilde{T})-c(t+\tilde{T})\le \overline{h}(t)-ct-c\tilde{T}$ and so $\hat{T}$ cannot be $\infty$.
\end{proof}

Now we give a sufficient condition for vanishing.
\begin{proposition}\label{suffvanish1}
There exists a function $\psi_{h_0,c,\mu}$ such that if $u_0(x)\le \psi_{h_0,c,\mu}(x)$ for $x\in [0,h_0]$, then there exists $\hat{T}>0$ such that $\lim_{t\nearrow \hat{T}}(h(t)-ct)=0$, where $(u, h)$ is 
the unique solution to problem \eqref{fbp2}.
\end{proposition}

\begin{proof}
Choose $C>0$ such that
\begin{align*}
2\sqrt{2K}C\mu\le c.
\end{align*}
For this $C$, we take $\varepsilon>0$ sufficiently small such that
\begin{align*}
\varepsilon<\min\left\{\frac{ch_0}{\mu\pi C}, Ce^{-\frac{2Kh_0}{c}},1\right\},\ 
\varepsilon^2\mu\left(\frac{c}{2}+\frac{\pi}{h_0}\right)<\frac{c}{4}.
\end{align*}
Now we consider the problem
\begin{align}\label{uppercos}
\left\{
\begin{array}{ll}
\zeta_t=\zeta_{xx}+K\zeta, &t>0, ct<x<h_0+ct, \\ 
\zeta(t, ct)=\zeta(t, h_0+ct)=0, &t>0, \\
\eta(0,x)=\phi_{h_0,c}(x),& 0\le x\le h_0,
\end{array}\right.
\end{align}
where $\psi_{h_0,c,\mu}(x)=\varepsilon^2e^{-\frac{c}{2}x}\sin\frac{\pi}{h_0}x$. Direct calculation and the choice of 
$\varepsilon$ yield that $\|\psi'_{h_0,c}\|_{C[0,h_0]}\le\frac{c}{4\mu}$. The solution of (\ref{uppercos}) is 
\begin{align*}
\zeta(t,x)=\varepsilon^2 e^{\left\{K-\left(\frac{\pi^2}{h_0^2}+\frac{c^2}{2}\right)\right\}t}e^{-\frac{c}{2}(x-ct)}
\sin\frac{\pi(x-ct)}{h_0}.
\end{align*}
Set
\begin{align*}
T:=\frac{1}{K}\log\frac{C}{\varepsilon}.
\end{align*}
By the choice of $\varepsilon$, $T>\frac{2h_0}{c}$ holds. Denote the solution of \eqref{fbp2} with initial data 
$u_0(x)=\psi_{h_0,c}(x)$ by $u(t,x)$. Now we can see that
\begin{align*}
-\mu\zeta_x(t, h_0+ct)\le \mu\varepsilon^2e^{Kt}\frac{\pi}{h_0}e^{-ch_0}\le\mu\varepsilon^2 e^{KT}\frac{\pi}{h_0}<c\ \ {\rm for}\ 
0\le t\le T, 
\end{align*}
by the choice of $\varepsilon$. We next note that
\begin{align*}
u(t, ct)=\zeta(t, ct)=0\ \ {\rm for}\ 0\le t\le T.
\end{align*}
Thus $(\zeta, ct+h_0)$ is an upper solution of (\ref{fbp2}) and 
\begin{align*}
u(t,x)\le \zeta(t,x)\le\varepsilon^2e^{KT}\le C\ \ {\rm for}\ ct\le x\le h(t), \ 0\le t\le T.
\end{align*}

Now we again consider the following function
\begin{align*}
\overline{U}(t,x)=-CM^2(x-h(t))(x-h(t)+2/M)
\end{align*}
over $\tilde{\Delta}_M:=\{(t,x)\in\mathbb{R}^2:0\le t<T,\ \max\{ct, h(t)-1/M, ct\}\le x\le h(t)\}$, where
\begin{align*}
M:=\max\left\{\frac{\sqrt{2K}}{2}, \frac{\|\psi_{h_0,c,\mu}'\|_{C[0,h_0]}}{C}\right\}.
\end{align*}
A direct calculation as in Lemma \ref{hprime} shows that $u(t,x)\le \overline{U}(t,x)$ on $Q$. So we have
\begin{align*}
-\mu u_x(t,h(t))\le-\mu \overline{U}_x(t, h(t))=2MC\mu \le\frac{c}{2}
\end{align*}
by choice of $C$. Thus we have
\begin{align*}
h'(t)-c=-\mu u_x(t,h(t))-c\le -\frac{c}{2}
\end{align*}
and
\begin{align*}
h(t)-ct=h_0+\int_0^t(h'(s)-c)\le h_0-\frac{c}{2}t\to 0\ \ {\rm as}\ t\to\frac{2h_0}{c}<T.
\end{align*}
Thus the solution $(u,h)$ can not be a global solution, that is vanishing happens for $(u,h)$. 
Therefore any solution of \eqref{fbp2} with the initial function less than $\psi_{h_0, c}$ also vanishes.
\end{proof}

\begin{corollary}\label{suffvanish2}
There exists a positive constant $C=C(h_0, c, \mu)$ such that if $\|u_0\|_{C[0, h_0]}\le C$, then there exists $\hat{T}>0$ such that $\lim_{t\nearrow \hat{T}}(h(t)-ct)=0$, where $(u, h)$ is the unique solution to \eqref{fbp2}.
\end{corollary}
\begin{proof}
Consider $\psi_{2h_0,c,\mu}(x)$ and define $C(h_0, c,\mu):=\inf_{x\in\left[\frac{1}{2}h_0, \frac{3}{2}h_0\right]}\psi_{2h_0,c,\mu}(x)>0$. Suppose that $u_0(x)$ satisfies  $u_0(x)\le C(h_0, c,\mu)$ 
for $x\in [0, h_0]$, then we have 
$u_0\left(x-\frac{h_0}{2}\right)\le \psi_{2h_0,c}(x)$ for $x\in[\frac{h_0}{2}, \frac{3}{2}h_0]$. Now we consider the following free boundary problem
\begin{align*}
\begin{cases}
u_t=u_{xx}+f(u),          &t>0,\ ct+\frac{h_0}{2}<x<h(t),\\
u\left(t,ct+\frac{h_0}{2}\right)=u(t,h(t))=0,      &t>0, \\
h'(t)=-\mu u_x(t,h(t)),   &t>0, \\
h(0)=\frac{1}{2}h_0,\ u(0,x)=u_0\left(x-\frac{h_0}{2}\right), &\frac{h_0}{2}\leq x\leq \frac{3}{2}h_0.
\end{cases}
\end{align*}
Denote $(u_1, h_1)$ the solution to above problem and $(u_2, h_2)$ the solution to the problem \eqref{fbp2} with initial function $\psi_{2h_0,c}$. By comparison principle, 
we have 
\begin{align*}
&u_1(t,x)\le u_2(t,x),\ {\rm for}\ t\in (0, T^*_1),\ \ x\in \left[ct+\frac{h_0}{2}, h_1(t)\right], \\ 
&h_1(t)\le h_2(t), {\rm for}\ t\in (0, T^*_1),
\end{align*}
where $T^*_1$ is the maximal existence time of solution $(u_1, h_1)$. By Lemma \ref{suffvanish1}, there exists $\hat{T}>0$ such that $\lim_{t\nearrow \hat{T}}(h_2(t)-ct)=0$. 
Hence there exists $\hat{T}_0>0$ such that $\lim_{t\nearrow\hat{T}_0}\left(h_1(t)-\left(ct+\frac{h_0}{2}\right)\right)=0$. Since the solution to \eqref{fbp2} with initial function $u_0$ is expressed by 
$\left(u_1\left(t,x+\frac{h_0}{2}\right), h_1(t)-\frac{h_0}{2}\right)$, we have obtained the conclusion.
\end{proof}

Finally we can obtain the following proposition.
\begin{proposition}\label{iff}
Let $(u,h)$ be the unique solution to problem \eqref{fbp2} on $[0, T^*)$ with maximal existence time $T^*$. Then $T^*<\infty$ if and only if $\lim_{t\nearrow T^*}(h(t)-ct)=0$.
\end{proposition}
\begin{proof}
By Proposition \ref{fte2}, we have that if $\lim_{t\nearrow T^*}(h(t)-ct)=0$ then $T^*<\infty$. 

Suppose that $T^*<\infty$. Then by Lemma \ref{hprime} we have that $\inf_{t\in (0,T^*)}(h(t)-ct)=0$ and there exists $\{t_n\}$ with $\lim_{n\to\infty}t_n=T^*$ such that 
$\lim_{n\to\infty}(h(t_n)-ct_n)=0$. By the proof of Lemma \ref{fte1}, we can show that $\lim_{n\to\infty}\|u(t_n,\,\cdot\,)\|_{C[ct_n, h(t_n)]}=0$. Then by the proof of Corollary \ref{suffvanish2} we can conclude 
that there exists $0<\hat{T}<\infty$ such that $\lim_{t\nearrow \hat{T}}(h(t)-ct)=0$. By Lemma \ref{hprime} again, $T^*=\hat{T}$ must be holds.
\end{proof}

\section{Proof of Theorem A}
In this section we will prove Theorem A. By Proposition \ref{iff}, if $T^*<\infty$, then the vanishing case in Theorem A happens. 
Therefore, to prove Theorem A, it suffices to prove the following theorem.
\begin{thm}
Suppose that $c\in(0, c^*)$ and $(u, h)$ is the unique solution of \eqref{fbp} defined for all $t>0$. Then either of the following occurs
\begin{enumerate}[{\rm (1)}]
\item $\lim_{t\to\infty}(h(t)-ct)=\infty$ and spreading happens;
\item $\lim_{t\to\infty}(h(t)-ct)=L_c$ and transition happens.
\end{enumerate}
\end{thm}
Throughout this section we assume that $0<c<c^*$ and $(u,h)$ is a global solution to \eqref{fbp}. 
Let $H_c(t):=h(t)-ct$. By Proposition \ref{fte2} we have $H_c(t)>0$ for any $t>0$.

\subsection{Some properties of $h(t)$}
\begin{lemma}\label{Hcinfty}
Suppose that $H_c(t)$ is unbounded, we have $\lim_{t\to\infty}H_c(t)=\infty$.
\end{lemma}
\begin{proof} One can prove this lemma by same approach as in \cite{DWZ}. For reader's convenience we give the proof of this lemma. We fix any large $\tilde{l}>h(0)=h_{0}$ and define 
\begin{align*}
w(t,x):=V_{c}(x-ct-\tilde{l}),\ \ t>0,\ ct+\tilde{l}\le x\le ct+\tilde{l}+L_{c},
\end{align*}
where $(L_{c}, V_{c})$ is the unique solution pair to problem \eqref{compact-support-wave}. $w$ satisfies
\begin{align*}
\left\{
\begin{array}{ll}
w_{t}=w_{xx}+w(1-w), & t>0,\ ct+\tilde{l}<x<ct+\tilde{l}+L_{c}, \\
w(t, ct+\tilde{l})=w(t, ct+\tilde{l}+L_{c})=0, & t>0, \\
w(0,x)=V_{c}(x-\tilde{l}), & \tilde{l}\le x\le \tilde{l}+L_{c}, \\
-\mu w_{x}(t, ct+\tilde{l}+L_{c})=c, & t>0.
\end{array}\right.
\end{align*}
Since $H_{c}(t)$ is unbounded, there exists $t_{1}>0$ such that $h(t_{1})-ct_{1}=\tilde{l}+L_{c}$. We can find $t_{2}\in (0, t_{1})$ such that 
$h(t_{2})-ct_{2}=\tilde{l}$ and $h(t)-ct\in(\tilde{l}, \tilde{l}+L_{c})$ for $t\in (t_{2}, t_{1})$. By the definition of $t_{1}$, $h'(t_{1})-c\ge 0$. Define
\begin{align*}
\eta(t,z):=u(t,x)-w(t,x).
\end{align*}
The function $\eta$ satisfies the following linear parabolic equation: 
\begin{align*}
\eta_t=\eta_{xx}+m(t,x)\eta,\ \ t>t_1,\ \ \tilde{l}+ct<x<\min\{\tilde{l}+ct+L_c,\ h(t)\},
\end{align*}
where $m(t,x)$ is some bounded function. Since $\eta(t, \tilde{l}+ct)=u(t, \tilde{l}+ct)>0$ and $\eta(t, h(t))=-w(t, h(t))<0$, we can apply Lemma \ref{zeronumber} to conclude that $\eta(t, \cdot)$ 
has finite number of zeros on $[ct+\tilde{l}, ct+\tilde{l}+L_{c}]$ for $t\in(t_{2}, t_{1})$. 
For $t$ just after $t_{2}$, by using the Hopf Lemma, $\eta(t,\cdot)$ have just one zero on $[\tilde{l}+ct, h(t)]$. By \cite{A}, zero number of $\eta(t, \cdot)$ on $[\tilde{l}+ct, h(t)]$ is nonincreasing, 
so $\eta(t, \cdot)$ has exactly one zero, say $z(t)$, on $[\tilde{l}+ct, h(t)]$ for $t\in (t_2, t_1)$. Moreover $z\in C^1(t_1, t_2)$.

We now claim that $\lim_{t\nearrow t_1}z(t)$ exists. Otherwise
\begin{align*}
\tilde{l}+ct_1\le \underline{z}:=\liminf_{t\nearrow t_1}z(t)<\limsup_{t\nearrow t_1}z(t)=:\overline{z}\le \tilde{l}+ct_1+L_c.
\end{align*}
It is easily seen that $\eta(t_1, \cdot)\equiv 0$ on $[\underline{z},\ \overline{z}]$. As in the proof in \cite{CLZ}, \cite{DLZ} and \cite{KM}, we may apply Theorem 2 in \cite{F} to $\eta$ 
over $[t_1-\varepsilon,\ t_1]\times [\tilde{l}+c(t_1+\varepsilon),\ h(t_1-\varepsilon)]$ with sufficiently small $\varepsilon>0$. By letting $\varepsilon\to 0$, we deduce $\eta(t_1, z)\equiv 0$ on 
$[\tilde{l}+ct_1,\ h(t_1)]$. However, this is impossible since $\eta(t_1, \tilde{l}+ct_1)>0$. Therefore $\lim_{t\nearrow t_1}z(t)$ exists. 

Next we claim that $z(t_1)=ct_1+\tilde{l}+L_c$. Suppose that $z(t_1)<ct_1+\tilde{l}+L_c$. Since $\eta$ satisfies 
\begin{align*}
&\eta(t,x)<0\ \ {\rm on}\ \ \{(t,x)\in\mathbb{R}^2\ |\ t\in(t_2, t_1),\ z(t)<x<h(t)\}, \\
&\eta(t_1, h(t_1))=0,
\end{align*}
we can apply the strong maximum principle and the Hopf Lemma to $\eta$ over $\{(t,x)\in\mathbb{R}^2 | t\in (t_2, t_1),\ z(t)<x<h(t)\}$ to conclude that $\eta_x(t_1, h(t_1))>0$. However this implies 
\begin{align*}
h'(t_1)=-\mu u_x(t_1, h(t_1))<-\mu w_x(t_1, h(t_1))=-\mu w_x(t_1, \tilde{l}+ct_1+L_c)=-\mu V_c'(L_c)=c
\end{align*}
which contradicts $h'(t_1)-c\ge 0$. So we obtain $u(t_1, x)\ge w(t_1, x)$ for $x\in [ct_1+\tilde{l},\ ct_1+\tilde{l}+L_c]$. By applying the strong maximum principle to $\eta$ over $\{(t,x)\in\mathbb{R}^2 | t\in (t_2,\ t_1),\ 
ct+\tilde{l}<x<z(t)\}$ we have 
\begin{align*}
u(t_1, x)>w(t_1, x)\ \ {\rm on}\ \ x\in[ct_1+\tilde{l}, ct_1+\tilde{l}+L_c].
\end{align*}
If we set $\xi(t):=ct+\tilde{l}$, $\underline{h}(t):=ct+\tilde{l}+L_c$ and $\underline{u}(t,x)=w(t,x)$, then $(\underline{u}, \underline{h})$ and 
$\xi$ satisfy the conditions of comparison principle (Lemma \ref{comp}) 
with initial time replaced by $t_1$. Thus we have 
\begin{align*}
&u(t,x)>w(t,x)\ \ {\rm for}\ \ t>t_1,\ \ x\in[ct+\tilde{l},\ ct+\tilde{l}+L_c], \\ 
&h(t)>\underline{h}(t)=ct+\tilde{l}+L_c\ \ {\rm for}\ \ t>t_1.
\end{align*}
Therefore we have $h(t)-ct>\tilde{l}+L_c$ for $t>t_1$. This means that $\lim_{t\to\infty}(h(t)-ct)=\infty$. 
\end{proof}

\begin{lemma}\label{Hcunbounded} 
If $H_c(t)$ is unbounded, then we have $\lim_{t\to\infty}\frac{h(t)}{t}=c^*$. 
\end{lemma}
\begin{proof} By Proposition \ref{upperestspeed}, we have
\begin{align*}
\limsup_{t\to\infty}\frac{h(t)}{t}\le c^*.
\end{align*}
Therefore it suffices to show that for any $\tilde{c}\in(c, c^*)$
\begin{align}\label{liminf}
\liminf_{t\to\infty}\frac{h(t)}{t}\ge \tilde{c}.
\end{align}
Although the proof is almost identical to the proof of Lemma 3.3 in \cite{DWZ}, we give the proof of this lemma for reader's convenience. 

Now fix $\tilde{c}\in(c,c^*)$ arbitrary. By Lemma \ref{compact-support-wave-ex}, there exists a unique solution pair $(L_{\tilde{c}}, V_{\tilde{c}})$ to problem \eqref{compact-support-wave}. 
It follows from the strong maximum principle that there exists $\varepsilon>0$ such that 
\begin{align*}
V_{\tilde{c}}(z)\le 1-\varepsilon\ \ {\rm for}\ \ z\in [0, L_{\tilde{c}}].
\end{align*}
By Lemma \ref{auxelliptic}, for any sufficiently large $l$, the problem
\begin{align*}
\left\{
\begin{array}{l}
w''+cw'+w(1-w)=0,\ -l<z<l, \\
w(-l)=w(l)=0,
\end{array}\right.
\end{align*}
has a unique positive solution $w_l$ and $w_l\to 1$ as $l\to\infty$ uniformly in any compact subset of $\mathbb{R}$. So, there exists $\tilde{l}>L_{\tilde{c}}$ such that 
\begin{align*}
w_{\tilde{l}}(z)>1-\frac{\varepsilon}{2}\ \ {\rm for}\ \ -\frac{L_{\tilde{c}}}{2}\le z\le \frac{L_{\tilde{c}}}{2}.
\end{align*}

Define $\tilde{\psi}(z)=w_{\tilde{l}}(z-\tilde{l})$. $\tilde{\psi}$ satisfies
\begin{align*}
\left\{
\begin{array}{l}
\tilde{\psi}''+c\tilde{\psi}'+\tilde{\psi}(1-\tilde{\psi})=0,\ 0<z<2\tilde{l}, \\
\tilde{\psi}(0)=\psi(2\tilde{l})=0.
\end{array}\right.
\end{align*}
and 
\begin{align}\label{psi}
\tilde{\psi}(z)>1-\frac{\varepsilon}{2}\ \ {\rm for}\ \ \tilde{l}-\frac{L_{\tilde{c}}}{2}\le z\le \tilde{l}+\frac{L_{\tilde{c}}}{2}.
\end{align}
Now we choose any $\psi_0\in C^1([0, 2\tilde{l}])$ satisfying $\psi_0(z)>0$ on $(0, 2\tilde{l})$ and $\psi_0(0)=\psi_0(2\tilde{l})=0$ and consider the following initial boundary 
value problem 
\begin{align*}
\left\{
\begin{array}{ll}
\psi_t=\psi_{zz}+c\psi_z+\psi(1-\psi), & t>0, 0<z<2\tilde{l}, \\
\psi(t, 0)=\psi(t, 2\tilde{l})=0, & t>0, \\
\psi(0, z)=\psi_0(z),& 0\le z\le 2\tilde{l}.
\end{array}\right.
\end{align*}
This problem has a unique positive solution $\psi(t,z; \psi_0)$ and it is well known that
\begin{align*}
\psi(t, z; \psi_0)\to\tilde{\psi}(z)\ \ {\rm as}\ t\to\infty\ \ {\rm uniformly\ on}\ [0, 2\tilde{l}].
\end{align*}
By \eqref{psi}, there exists $T=T(\psi_0)$ such that 
\begin{align}\label{psi1}
\psi(t, z; \psi_0)>1-\varepsilon\ \ {\rm for}\ t>T\ \ {\rm and}\ \ \tilde{l}-\frac{L_{\tilde{c}}}{2}<x<\tilde{l}+\frac{L_{\tilde{c}}}{2}.
\end{align}
Since $\lim_{t\to\infty}H_c(t)=\infty$ by Lemma \ref{Hcunbounded}, we can find $T_1>0$ such that $H_c(t)>2\tilde{l}$ for all $t\ge T_1$. Now we define $v(t,z)=u(t, z+ct)$. 
Then we have
\begin{align*}
\left\{
\begin{array}{ll}
v_t=v_{zz}+cv_z+v(1-v), & t>T_1,\ 0<z<H_c(t), \\
v(t,0)=v(t,H_c(t))=0, & t>T_1, \\
v(T_1, z)=u(T_1, z+cT_1), & 0\le z\le H_c(T_1).
\end{array}\right.
\end{align*}
Therefore, if we choose $\psi_0$ in \eqref{psi1} satisfying $0<\psi_0(z)\le u(T_1, z+cT_1)$ for $0\le z\le 2\tilde{l}$, then by using the standard comparison principle we obtain
\begin{align*}
\psi(t,z;\psi_0)<v(t+T_1, z)\ \ {\rm for}\ \ t>0\ \ {\rm and}\ \ 0<z<2\tilde{l}.
\end{align*}
By \eqref{psi} we have
\begin{align}\label{psi2}
v(T_1+T, z)>\psi(T, z;\psi_0)>1-\varepsilon\ \ {\rm for}\ \ \tilde{l}-\frac{L_{\tilde{c}}}{2}<z<\tilde{l}+\frac{L_{\tilde{c}}}{2}.
\end{align}
Denote $T_0=T+T_1$. \eqref{psi2} implies that 
\begin{align*}
u(T_0, x)>1-\varepsilon\ \ {\rm for}\ \ \tilde{l}+cT_0-\frac{L_{\tilde{c}}}{2}<x<\tilde{l}+cT_0+\frac{L_{\tilde{c}}}{2}.
\end{align*}
Now we define
\begin{align*}
&\underline{u}(t,x):=V_{\tilde{c}}(x-\underline{h}(t)), \\ 
&\underline{h}(t):=\tilde{c}(t-T_0)+\tilde{l}+cT_0+\frac{L_{\tilde{c}}}{2}, \\
&\xi(t):=\tilde{c}(t-T_0)+\tilde{l}+cT_0-\frac{L_{\tilde{c}}}{2}.
\end{align*}
It is easily seen that
\begin{align*}
\left\{
\begin{array}{ll}
\underline{u}_t=\underline{u}_{xx}+\underline{u}(1-\underline{u}), & t>T_0,\ \ \xi(t)<x<\underline{h}(t), \\
\underline{u}(t, \xi(t))=\underline{u}(t, \underline{h}(t))=0, & t>T_0, \\
\underline{h}'(t)=\tilde{c}=-\mu\underline{u}_x(t, \underline{h}(t)), & t>T_0,
\end{array}\right.
\end{align*}
and 
\begin{align*}
&ct<\xi(t)\ \ {\rm for}\ \ t\ge T_0, \\
&\underline{u}(T_0, x)=V_{\tilde{c}}\left(x-\tilde{l}-cT_0+\frac{L_{\tilde{c}}}{2}\right)<1-\varepsilon<u(T_0,x)\ \ {\rm for}\ \ \xi(T_0)\le x\le \underline{h}(T_0).
\end{align*}
By Lemma \ref{comp}, we obtain
\begin{align*}
&u(t,x)\ge \underline{u}(t,x)\ \ {\rm for}\ \ t\ge T_0,\ \ x\in[\xi(t), \underline{h}(t)], \\
&h(t)\ge \underline{h}(t)=\tilde{c}(t-T_0)+\tilde{l}+cT_0+\frac{L_{\tilde{c}}}{2}\ \ {\rm for}\ \ t\ge T_0.
\end{align*}
This implies \eqref{liminf}.
\end{proof}

\begin{proposition}\label{limit}
If $H_c(t)$ is bounded, then $\lim_{t\to\infty}H_c(t)$ exists.
\end{proposition}
The next lemma is sufficient to prove Proposition \ref{limit}.
\begin{lemma}\label{sign}
For any $b\in(0,\infty)\backslash\{L_c\}$, $H_c(t)-b$ changes its sign at most finitely many times. 
\end{lemma}
\begin{proof}
Define $v(t,z):=u(t, z+ct)$. It is clear that $(v, H_c)$ satisfies
\begin{align*}
\left\{
\begin{array}{ll}
v_t=v_{zz}+cv_z+v(1-v), & t>0,\ 0<z<H_c(t), \\
v(t,0)=v(t, H_c(t))=0, & t>0, \\
H_c'(t)=-\mu v_x(t, H_c(t))-c, & t>0, \\
H_c(0)=h_0, v(0,z)=u_0(z),& 0\le z\le h_0.
\end{array}\right.
\end{align*}
As in the proof of Lemma 3.7 in \cite{KM}, we investigate the zero number of the function $v(t, z)-V_c(z-b+L_c)$ for any $b\in (0, \infty)\backslash \{L_c\}$.

{\bf Step 1.} For the case $0<b<L_c$.

{\bf Case 1.} We first consider the case where $H_c(0)=h_0<b$. If $u_0(x)\le V_c(x-b+L_c)$, then by the comparison principle (Lemma \ref{comp}), we have $u(t,x)\le V_c(x-ct-b+L_c)$ and 
$h(t)\le b+ct$ for $t>0$ and $x\in [ct, h(t)]$. Thus $H_c(t)\le b$ for $t\ge 0$. By applying the strong maximum principle for $u(t,x)-V_c(x-ct-b+L_c)$ over $\{(t,x)\in\mathbb{R}^2\ |\ t>0, x\in [ct, h(t)]\}$, 
we have $u(t,x)<V_c(x-ct-b+L_c)$ in the region. Furthermore, we can show that $H_c(t)<b$ for $t>0$. Otherwise $H_c(t^*)=b$ holds for some $t^*>0$, and then we have $H_c'(t^*)\ge 0$. However, 
since $h(t^*)-ct^*=b$ and $u(t^*, h(t^*))-V_c(u(t^*)-ct^*-b+L_c)=u(t^*, h(t^*))-V_c(L_c)=0$, so we can apply the Hopf Lemma to obtain $u_x(t^*, h(t^*))>V_c'(L_c))$. But this means that 
$h'(t^*)=-\mu u_x(t^*, h(t^*))<-\mu V_c'(L_c)=c$ which contradicts $H'(t^*)\ge 0$. Hence we have $H_c(t)<b$ for $t>0$.

Now we assume that $u_0(X)>V_c(X-b+L_c)$ for some $X\in(0,h_0)$. Define 
\begin{align*}
\eta(t,z):=v(t,z)-V_c(z-b+L_c).
\end{align*}
The function $\eta$ satisfies the following linear parabolic equation:
\begin{align*}
\eta_t=\eta_{zz}+cv_z+m(t,z)\eta,\ \ t>0,\ \ 0<z<k(t),
\end{align*}
where $k(t):=\min\{H_c(t), b\}$ and $m(t,z)$ is certain bounded function. Now we note that $\eta(0,X)>0$ and $\eta(0, h_0)<0$. Suppose that there exists $t_1>0$ such that 
\begin{align*}
H_c(t)<b\ \ {\rm for}\ \ t\in[0, t_1)\ \ {\rm and}\ \ H(t_1)=b.
\end{align*}
This implies that $H'(t_1)\ge 0$. Since $\eta(t, 0)<0$ for $t>0$ and $\eta(t, k(t))<0$ for $t\in(0, t_1)$, we can apply the result of the zero number argument from Lemma \ref{zeronumber} for $t\in(0,t_1)$. 
Let $\mathcal{Z}(t)$ be the number of zeros of the function $\eta(t,\cdot)$ in the closed interval $[0, k(t)]$. By Lemma \ref{zeronumber}, we have $\mathcal{Z}(t)<\infty$ for each $t\in(0,t_1)$. Moreover , 
$\eta(t, \cdot)$ can have degenerate zeros at most finitely many values of $t$ in $(0, t_1)$. Thus, we can find $\tau_0\in (0, t_1)$ such that for $t\in(\tau_0, t_1)$, $\eta(t, \cdot)$ has only nondegenerate zeros 
$\{z_j(t)\}_{j=1}^m$ with
\begin{align*}
0<z_1(t)<\cdots<z_m(t)<k(t).
\end{align*}
We note that $z_j(\cdot)\in C^1(\tau_0, t_1)$ for $j=1,\cdots, m$. As in the proof of Lemma \ref{Hcinfty}, we can show that $z_j^*:=\lim_{t\nearrow t_j}z_j(t)$ exists for each $j=1,\cdots, m$. 

{\bf Claim 1.} $z_m^*=b$.

Assume that $z_m^*<b$. Then, by applying the strong maximum principle to the function $\eta$ over $\{(t,z)| \tau_0<t\le t_1,\ z_m(t)<z<k(t)\}$, we obtain $\eta(t,z)<0$. Since $\eta(t, k(t_1))=0$, 
we can use the Hopf Lemma to deduce that $\eta_z(t_1, k(t_1))=\eta_z(t_1, b)>0$. This implies $H_c'(t_1)=-\mu v_z(t_1, b)-c<-\mu V_c'(L_c)-c=0$, which contradicts $H'(t_1)\ge 0$. Hence $z_m^*=b$. 

{\bf Claim 2.} If $z_j^*<z_{j+1}^*$, then $\eta(t_1, z)\ne 0$ for $z\in(z_j^*, z_{j+1}^*)$. This follows by applying the strong maximum principle to the function $\eta$ over 
$\{(t, z)|\tau_0<t\le t_1,\ z_j(t)<z<z_{j+1}(t)\}$. 

From Claim 1 and 2, we can see that $n:=\mathcal{Z}(t_1)\le m=\mathcal{Z}(t)$ for $t\in (\tau_0, t_1)$. Let $0<\hat{z}_1<\cdots<\hat{z}_n=k(t_1)=b$ denote all the zeros of $\eta(t_1, \cdot)$ 
in $[0, k(t_1)]$.

Next we will show that there exists $\varepsilon>0$ such that $\mathcal{Z}(t_1)>\mathcal{Z}(t)$ for $t\in(t_1, t_1+\varepsilon)$.

{\bf Claim 3.} The zero $\hat{z}_n(=b)$ of $\eta(t_1, \cdot)$ disappears just after $t_1$.

Take $\tilde{b}<b$ such that $\eta(t_1, z)\ne 0$ for $z\in [\tilde{b}, b)$. For definiteness, we assume that $\eta(t_1, z)>0$ for $z\in [\tilde{b}, b)$. By continuity, we can choose 
a sufficiently small $\varepsilon>0$ such that $\eta(t,\tilde{b})>0$ for $t\in [t_1, t_1+\varepsilon]$. Let us define
\begin{align*}
&\xi(t):=\tilde{b}+ct,\ \underline{h}(t):=b+ct, \\
&\underline{u}(t,x):=V_c(x-ct-b+L_c),
\end{align*}
for $t\in [t_1, t_1+\varepsilon]$, $x\in[\tilde{b}+ct, b+ct]$. Then, it is easy to see that $(\underline{u}, \underline{h})$ is a lower solution of \eqref{fbp} for such $t$ and $x$. By the comparison 
principle (Lemma \ref{comp}), we obtain $u(t,x)\ge V_c(x-ct-b+L_c)$ and $h(t)\ge b+ct$ for $t\in[ t_1, t_1+\varepsilon]$ and $x\in[\tilde{b}+ct, b+ct]$, so we have $H_c(t)\ge b$ for 
$t\in[t_1, t_1+\varepsilon]$. By the strong maximum principle, we obtain $u(t,x)>V_c(x-ct-b+L_c)$ for $t\in (t_1, t_1+\varepsilon]$ and $x\in[\tilde{b}+ct, b+ct)$ or $\eta(t,z)=v(t,z)-V_c(z-b+L_c)>0$ 
for $t\in (t_1, t_1+\varepsilon]$ and $z\in[\tilde{b}, b)$. We can also show that $H_c(t)>b$ for $t\in (t_1, t_1+\varepsilon]$. In fact, if $H_c(\tilde{t})=b$ for some $\tilde{t}\in(t_1, t_1+\varepsilon]$, then 
we obtain 
$H_c'(\tilde{t})\le 0$ and $\eta(\tilde{t}, b)=0$. So we can apply the Hopf Lemma to deduce that $\eta_z(\tilde{t}, b)<0$. This leads to $H_c'(\tilde{t})=-\mu v_z(\tilde{t}, H_c(\tilde{t}))-c>0$, which 
is a contradiction. Thus $H_c(t)>b$ for $t\in (t_1, t_1+\varepsilon]$. This means that the zero $\hat{z}_n(=b)$ disappears just after $t_1$. Even in the case where $\eta(t_1, z)<0$ for 
$z\in[\tilde{b}, b)$, we can show that $v(t,z)<V_c(z-b+L_c)$ and $H_c(t)<b$ for $t\in (t_1, t_1+\varepsilon]$ and $z\in [\tilde{b}, H_c(t))$, and the zero $\hat{z}_n(=b)$ disappears just after $t_1$.

On the other hand, since we can see that $\eta(t, \tilde{b})\ne 0$ for $t\in [t_1-\varepsilon, t_1+\varepsilon]$ by shrinking $\varepsilon$, the number of zeros of $\eta(t, \cdot)$ in $[0, \tilde{b}]$
 is nonincreasing on $[t_1-\varepsilon, t_1+\varepsilon]$. Therefore, we can deduce that $\mathcal{Z}(s)\ge\mathcal{Z}(t_1)>\mathcal{Z}(t)$ for $s\in (t_1-\varepsilon, t_1)$ and $t\in(t_1, t_1+\varepsilon)$.
 
 {\bf Case 2.} Next, we consider the case where $b<h_0$. Define $\eta$, $k(t)$ and $\mathcal{Z}(t)$ as given in Case 1. Since $\eta(0,0)=u_0(0)-V_c(-b+L_c)<0$ and $\eta(0, b)=u_0(b)-V_c(L_c)>0$, 
 $\eta(0,z)$ has at least one zero on $[0,b]$. If $H_c(t)>b$ for $t>0$, then nothing more is required. Suppose that there exists $t_2>0$ such that 
 \begin{align*}
 b<H_c(t)\ \ {\rm for}\ \ t\in [0, t_2)\ \ {\rm and}\ \ H_c(t_2)=b.
 \end{align*}
 Then, we have $H_c'(t_2)\le 0$. By Lemma \ref{zeronumber}, we have $\mathcal{Z}(t)<\infty$ for each $t\in (0, t_2)$. In a similar way to Case 1, we can show that $\mathcal{Z}(t)$ decreases strictly when $t$ goes across $t_2$.
 
 {\bf Case 3.} We consider the case where $h_0=b$. If $H_c(t)\equiv h_0$, then nothing more is required. Assume that $H_c(\tau_2)>h_0$ or $H_c(\tau_2)<h_0$ for some $\tau_2>0$. Then, we 
 can regard $\tau_2$ as an initial time and obtain the same conclusion.
 
 Summarizing the arguments in Cases 1 -- 3, we can conclude that when $H_c(t)$ reaches $b$, the number of zeros of $\eta(t, \cdot)$ decreases strictly. Since $\mathcal{Z}(t_0)<\infty$ for $t_0$ just after the initial time, unless
  $H_c(t)\equiv b$, we can conclude that $H_c(t)-b$ changes sign at most finitely many times after $t_0$. 
  
{\bf Step 2.} For the case $L_c<b$.

Define $\eta$ as in Step 1. By considering $\eta$ over the region $\{(t, z)| t>0, b-L_c<z<k(t)\}$, we can repeat the argument in Step 1 and obtain the same conclusion.

Now we have complete the proof of Lemma \ref{sign}.
\end{proof}

\begin{proof}[Proof of Proposition \ref{limit}]
Since $H_{c}(t)>0$ for all $t>0$ and $H_{c}(t)$ is bounded, there exist $\{t_{n}\}$, $\{\tilde{t}_{n}\}\subset\mathbb{R}$ with $\lim_{n\to\infty}t_{n}=\lim_{n\to\infty}\tilde{t}_{n}=\infty$ such that
\begin{align}\label{seq0}
0\le\liminf_{n\to\infty}H_{c}(t_n)=\lim_{t\to\infty}H_{c}(t)\le \limsup_{t\to\infty}H_{c}(t)=\lim_{n\to\infty}H_{c}(\tilde{t}_{n})<\infty
\end{align}
Suppose that $\underline{H}:=\liminf_{t\to\infty}H_{c}(t)<\limsup_{t\to\infty}H_{c}(t)=:\overline{H}$. Then \eqref{seq0} means that for $b\in (\underline{H}, \overline{H})\cap\{(0, \infty)\backslash\{L_{c}\}\}$, $H_{c}(t)-b$ changes its sign infinitely many times. But this contradict the conclusion of Lemma \ref{sign}. 
Now we have completed the proof of Proposition \ref{limit}.
\end{proof}

\begin{proposition}\label{HctoLc}
Suppose that $H_c(t)$ is bounded. Then we have $\lim_{t\to\infty}H_c(t)=L_c$.
\end{proposition}
\begin{proof}

Let $H_{c}^*:=\lim_{t\to\infty}H_{c}(t)$.

{\bf Step 1.} Suppose that $H_{c}^{*}<L_{c}$. Define
\begin{align*}
v(t,z):=u(t,z+ct),\ w(t,y):=u(t, y+h(t)).
\end{align*}
It is clear that $v$ and $w$ satisfy
\begin{align}
&\left\{
\begin{array}{ll}
v_{t}=v_{zz}+cv_{z}+v(1-v), &t>0,\ 0<z<H_{c}(t), \\
v(t,0)=0, & t>0,
\end{array}\right. \label{veq1} \\
&\left\{
\begin{array}{ll}
w_{t}=w_{yy}+(c+H_{c}'(t))w_{y}+w(1-w),& t>0,\ -H_{c}(t)<y<0, \\
w(t, -H_{c}(t))=w(t,0)=0,& t>0, \\
H_{c}'(t)=-\mu w_{y}(t,0)-c,& t>0.
\end{array}\right. \label{weq1}
\end{align}
Now we take any sequence $\{t_n\}\subset\mathbb{R}$ satisfying $\lim_{n\to\infty}t_{n}=\infty$ and define
\begin{align*}
H_{c,n}(t);=H_{c}(t+t_{n}),\ v_{n}(t, z):=v(t+t_{n}, z),\ w_{n}(t, y):=w(t+t_{n}, y).
\end{align*}
From \eqref{veq1}, \eqref{weq1}, we have
\begin{align}
&\left\{
\begin{array}{ll}
\displaystyle\frac{\partial v_{n}}{\partial t}=\frac{\partial^{2}v_{n}}{\partial t^{2}}+c\frac{\partial v_{n}}{\partial z}+v(1-v), &t>0, 0<z<H_{c,n}(t), \\
v_{n}(t,0)=0, & t>0,
\end{array}\right. \label{veq2} \\
&\left\{
\begin{array}{ll}
\displaystyle\frac{\partial w_{n}}{\partial t}=\frac{\partial^{2}w_{n}}{\partial y^{2}}+(c+H_{c,n}'(t))\frac{\partial w_{n}}{\partial y}+w(1-w),& t>-t_{n}, -H_{c,n}(t)<y<0, \\
w_{n}(t, -H_{c,n}(t))=w_{n}(t,0)=0,& t>-t_{n}, \\
\displaystyle H_{c,n}'(t)=-\mu \frac{\partial w_{n}}{\partial y}(t,0)-c,& t>-t_{n}.
\end{array}\right. \label{weq2}
\end{align}
We first examine \eqref{weq2}. Since $\|w_n\|_{\infty}$ and $\|H_{c,n}'\|_{\infty}$ are bounded, we can apply the parabolic $L^{p}$ estimates, Sobolev embedding theorem and the Schauder estimates (see \cite{Li} and \cite{LSU}) to deduce that $\{w_{n}\}$ is bounded in 
$C^{1+\frac{\alpha}{2}, 2+\alpha}([-R,R]\times[-H_{c}^{*}+\frac{1}{R}, 0])$ for any $R>0$ and $0<\alpha<1$. Hence $H_{c,n}'$ is uniformly bounded in $C^{\alpha}(I)$ for any bounded interval $I\subset\mathbb{R}$, and by passing to a subsequence, which is still denoted by $\{t_{n}\}$, 
we have
\begin{align*}
H_{c,n}'\to \tilde{H}_{c}\ \ {\rm in}\ \ C^{\alpha'}_{\rm loc}(\mathbb{R})\ \ {\rm as}\ \ n\to\infty
\end{align*}
for some function $\tilde{H}$ and any $\alpha'\in(0,\alpha/2)$. By passing to a further subsequence, we have
\begin{align*}
w_{n}\to \hat{w}\ \ {\rm in}\ \ C^{1+\frac{\alpha'}{2}, 2+\alpha'}_{\rm loc}(\mathbb{R}\times(-H_{c}^{*}, 0])\ \ {\rm as}\ \ n\to\infty
\end{align*}
and $\hat{w}$ satisfies
\begin{align*}
\left\{
\begin{array}{ll}
\hat{w}_{t}=\hat{w}_{yy}+(\tilde{H}_{c}+c)\hat{w}_{y}+\hat{w}(1-\hat{w}), & t\in\mathbb{R},\ -H_{c}^{*}<y<0, \\
\hat{w}(t,0)=0, & t\in\mathbb{R}, \\
\tilde{H}_{c}(t)=-\mu\hat{w}_{y}(t,0)-c, & t\in\mathbb{R}.
\end{array}\right. 
\end{align*}
Since
\begin{align*}
H_{n,c}(t)=H_{n,c}(0)+\int_{0}^{t} H_{n,c}'(s)ds,
\end{align*}
$H_{n,c}(t)=H_{c}(t+t_{n})$, $\lim_{t\to\infty}H_{c}(t)=H_{c}^{*}$ and $\lim_{n\to\infty}H_{c,n}'(t)=\tilde{H}(t)$ in $C^{\alpha'}_{\rm loc}(\mathbb{R})$, by letting $n\to\infty$ in the above identity, we find that
\begin{align*}
\int_{0}^{t} \tilde{H}(s)ds=0\ \ {\rm for\ all}\ \ t\in\mathbb{R},
\end{align*}
that is, $\tilde{H}(t)\equiv 0$. Therefore, we obtain
\begin{align*}
\left\{
\begin{array}{ll}
\hat{w}_{t}=\hat{w}_{yy}+c\hat{w}_{y}+\hat{w}(1-\hat{w}), &t\in\mathbb{R}, -H_{c}^{*}<y<0, \\
\hat{w}(t,0)=0, & t\in\mathbb{R}, \\
\displaystyle\frac{\partial\hat{w}}{\partial y}(t,0)=-\frac{c}{\mu}, & t\in\mathbb{R}.
\end{array}\right.
\end{align*}

Next we examine $v_{n}$. For any small $\varepsilon>0$, we consider \eqref{veq2} over
\begin{align*}
\Omega_{\varepsilon}:=\{(t,z) : t\in [-\varepsilon^{-1}, \varepsilon^{-1}], z\in[0, H_{c}^{*}-\varepsilon]\}.
\end{align*}
Applying the parabolic $L^{p}$ estimates, the Sobolev embedding theorem and the Schauder estimates, along a subsequence, we can show that $v_{n}\to\hat{v}$ in 
$C^{1+\alpha'/2, 2+\alpha'}(\Omega_{\varepsilon})$ as $n\to\infty$ for $\alpha'\in(0,1)$ and $\hat{v}$ satisfies
\begin{align*}
\hat{v}_{t}=\hat{v}_{zz}+c\hat{v}_{z}+\hat{v}(1-\hat{v})\ \ {\rm in}\ \ \Omega_{\varepsilon}.
\end{align*}
Since $\varepsilon>0$ is arbitrary, by using diagonal argument along a further subsequence, we obtain
\begin{align*}
v_{n}\to\hat{v}\ \ {\rm in}\ \ C^{1+\frac{\alpha'}{2}, 2+\alpha'}(\Omega_{0})
\end{align*}
where $\Omega_{0}:=\{(t,z): t\in\mathbb{R},\ z\in[0, H_{c}^{*})\}$. From the relation $v_{n}(t,z)=w_{n}(t, z-H_{c,n}(t))$, we have $\hat{v}(t,z)=
\hat{w}(t, z-H_{c}^{*})$ for $0<z<H_{c}^{*}$. Since $\hat{v}(t,0)=0$, we can easily see that
\begin{align*}
\lim_{y\to -H_{c}^{*}}\hat{w}(t,y)=\lim_{y\to -H_{c}^{*}}\hat{v}(t, y+H_{c}^{*})=0.
\end{align*}
So we have $\hat{w}\in C^{1,2}(\mathbb{R}\times [-H_{c}^{*}, 0])$ and
\begin{align}\label{weq4}
\left\{
\begin{array}{ll}
\hat{w}_{t}=\hat{w}_{yy}+c\hat{w}_{y}+\hat{w}(1-\hat{w}), & t\in\mathbb{R}, -H_{c}^{*}<y<0, \\
\hat{w}(t, -H_{c}^{*})=\hat{w}(t,0)=0, & t\in\mathbb{R}, \\
\displaystyle\frac{\partial\hat{w}}{\partial y}(t,0)=-\frac{c}{\mu}.
\end{array}\right.
\end{align}
By the strong maximum principle, we also have $\hat{w}(t, y)>0$ for $t\in\mathbb{R}$ and $y\in (-H_{c}^{*}, 0)$. 

Now we define $\eta(t,y)=\hat{w}(t,y)-V_{c}(y+L_{c})$. Clearly $\eta$ satisfies
\begin{align*}
&\eta_{t}=\eta_{yy}+c\eta_{y}+m(t,y)\eta, t\in\mathbb{R}, y\in[-H_{c}^{*}, 0], \\
&\eta(t, -H_{c}^{*})<0,\ \eta(t,0)=0.
\end{align*}
Therefore we can use the zero number result of Angenent \cite{A}(see Lemma \ref{zeronumber0}) to conclude that, for any $t\in\mathbb{R}$, the number 
of zeros of $\eta(t, \cdot)$ in $[-H_{c}^{*}, 0]$, say $\mathcal{Z}_{[-H_{c}^{*}, 0]}(t)$, is finite and nonincreasing in $t$, and if $\eta(t_{0}, \cdot)$ has a degenerate zero in $[-H_{c}^{*},0]$ for some $t_{0}\in\mathbb{R}$, then for any $s<t_{0}<t$ we have 
\begin{align*}
\mathcal{Z}_{[-H_{c}^{*},0]}(t)\le \mathcal{Z}_{[-H_{c}^{*}, 0]}(s)-1.
\end{align*}
Since $\mathcal{Z}_{[-H_{c}^{*}, 0]}(t)<\infty$, it follows that there may be at most finitely many value of $t$ such that $\eta(t, \cdot)$ has a degenerate zero. However $\eta$ satisfies
\begin{align*}
\eta_{y}(t,0)=\hat{w}_{y}(t, 0)-V_{c}'(L_{c})=0,
\end{align*}
so $\eta(t, \cdot)$ has  degenerate zero $y=0$ for any $t\in\mathbb{R}$. This is contradiction. Thus we have $L_{c}\le H_{c}^{*}$.

{\bf Step 2.} Suppose that $L_{c}<H_{c}^{*}$. Arguing as in Step 1, we obtain $\hat{w}$ satisfying \eqref{weq4} and $\hat{w}(t,y)>0$ for 
$t\in\mathbb{R}$ and $y\in (-H_{c}^{*}, 0)$. Noting that $L_{c}<H_{c}^{*}$, we consider $\eta(t,y)$ on $\{(t,y) : t\in\mathbb{R}, y\in [-L_{c}, 0]\}$. 
Then we have $\eta(t, -L_{c})>0$ and we can obtain a contradiction by similar zero number argument to Step 1. The proof is complete.
\end{proof}

\subsection{The case of spreading}

In this subsection, we investigate the spreading phenomena. We first give a sufficient condition for spreading.

\begin{lemma}\label{suffforspreading}
Suppose that
\begin{align*}
h_0\ge b+L_c\ \ {\rm and}\ \ u_0(x)\ge V_c(x-b)\ \ {\rm and}\ \ u_0(x)\not\equiv V_c(x-b)\ \ {\rm for}\ \ b\le x\le b+L_c
\end{align*}
for some $b\ge 0$. Then $\lim_{t\to\infty}H_c(t)=\infty$.
\end{lemma}
\begin{remark}
We need additional condition $u_0(x)\not\equiv V_c(x-b)$ which do not need for Lemma 3.7 in \cite{DLZ}, since $(V_c(\cdot -ct), ct+L_c)$ is solution to \eqref{fbp}.
\end{remark}
\begin{proof}[Proof of Lemma \ref{suffforspreading}]
Although the proof is almost same as the proof of Lemma 3.7 of \cite{DWZ}, we give the proof for reader's convenience.

Let us define
\begin{align*}
\xi(t):=b+ct,\ \underline{h}(t):=b+ct+L_c
\end{align*}
and
\begin{align*}
\underline{u}(t,x):=V_c(x-ct-b).
\end{align*}
Then we can use the comparison principle (Lemma \ref{comp}) to obtain
\begin{align*}
&\underline{u}(t, x)\le u(t,x)\ \ {\rm for}\ \ t>0,\ \ b+ct\le x\le b+ct+L_c, \\
&\underline{h}(t)\le h(t)\ \ {\rm for}\ \ t>0.
\end{align*}
Since $u_0(x)\not\equiv V_c(x-b)$, by the strong maximum principle and the Hopf Lemma, we obtain
\begin{align*}
&u(t,x)>V_c(x-ct-b)\ \ {\rm for}\ \ t>0,\ \ b+ct<x\le b+ct+L_c, \\
&h(t)>b+ct+L_c\ \ {\rm for}\ \ t>0.
\end{align*}
We now fix $t_0>0$. By the Hopf Lemma and the continuity of $(L_c, V_c)$ on $c$, we can find $\tilde{c}>c$ which is sufficiently close to $c$ such that
\begin{align*}
&h(t_0)\ge b+\tilde{c}t_0+L_{\tilde{c}}, \\
&u(t_0, x)\ge V_{\tilde{c}}(x-\tilde{c}t_0-b),\ \ {\rm for}\ \ b+\tilde{c}t_0<x<b+\tilde{c}t_0+L_{\tilde{c}}.
\end{align*}
We can use the comparison principle again to deduce that
\begin{align*}
&h(t)\ge b+\tilde{c}t+L_{\tilde{c}}\ \ {\rm for}\ \ t\ge t_0, \\
&u(t, x)\ge V_{\tilde{c}}(x-\tilde{c}t-b)\ \ {\rm for}\ \ t\ge t_0,\ \ b+\tilde{c}t_0\le x<b+\tilde{c}t_0+L_{\tilde{c}}.
\end{align*}
Hence we obtain $H_c(t)\ge b+(\tilde{c}-c)t+L_{\tilde{c}}$ and $\lim_{t\to\infty}H_c(t)=\infty$. Now we have completed the proof.
\end{proof}

\begin{proposition}\label{spreadingconv}
If $H_c(t)$ is unbounded, then $\lim_{t\to\infty}\frac{h(t)}{t}=c^*$ and for any given small $\varepsilon>0$ 
\begin{align*}
\lim_{t\to\infty}\max_{x\in [(c+\varepsilon)t, (c^*-\varepsilon)t]}|u(t, x)-1|=0.
\end{align*}
\end{proposition}
We can prove this proposition by the same way to the proof of Theorem 3.9 of \cite{DWZ}. The proof is a little bit technical, so we give the detail of the proof 
in Appendex for reader's convenience.

\subsection{The case of transition}

In this subsection, we prove following theorem.

\begin{thm}If $H_c(t)$ is bounded, then $\lim_{t\to\infty}H_c(t)=L_c$ and 
\begin{align}\label{transitionconv}
\lim_{t\to\infty}\left\{\sup_{x\in [ct, h(t)]}|u(t,x)-V_c(x-h(t)+L_c)|\right\}=0.
\end{align}
\end{thm}
\begin{proof}
The first assertion has been proved in Proposition \ref{limit}. We will prove the second assertion. Define
\begin{align*}
&v(t,z):=u(t, z+ct)\ \ {\rm for}\ \ t>0, 0<z<H_c(t), \\
&w(t,y):=u(t, y+h(t))\ \ {\rm for}\ \ t>0,\ -H_c(t)<y<0.
\end{align*}
Take any sequence $\{t_n\}$, satisfying $\lim_{n\to\infty}t_n=\infty$, and define
\begin{align*}
H_{c,n}(t):=H_c(t+t_n),\ v_n(t, z):=v(t+t_n, z),\ w_n(t, y):=w(t+t_n, y).
\end{align*}
Then $H_{c,n}$, $v_n$ and $w_n$ satisfy \eqref{veq2} and \eqref{weq2}. By the same argument as in Proposition \ref{HctoLc}, we can see that
 for any $\alpha\in(0,1)$, there exist a subsequence of $\{t_n\}$, functions $\hat{w}$ and $\hat{v}$ such that 
 \begin{align*}
 &H_{c,n}'\to 0 \ \ {\rm in}\ \ C^{\alpha}_{\rm loc}(\mathbb{R}), \\
 &w_n\to\hat{w}\ \ {\rm in}\ \ C^{1+\frac{\alpha}{2}, 2+\alpha}_{\rm loc}(\mathbb{R}\times (-L_c, 0]), \\
 &v_n\to\hat{v}\ \ {\rm in}\ \ C^{1+\frac{\alpha}{2}, 2+\alpha}_{\rm loc}(\mathbb{R}\times [0, L_c)), 
 \end{align*}
 along the subsequence, and $\hat{v}$, $\hat{w}$ satisfies
 \begin{align*}
 &\left\{
 \begin{array}{ll}
 \hat{v}_t=\hat{v}_{zz}+c\hat{v}_z+\hat{v}(1-\hat{v}), & t\in\mathbb{R},\ z\in [0, L_c), \\
 \hat{v}(t,0)=0 & t\in\mathbb{R},
 \end{array}\right. \\
 &\left\{
 \begin{array}{ll}
 \hat{w}_t=\hat{w}_{yy}+c\hat{w}_y+\hat{w}(1-\hat{w}), & t\in\mathbb{R},\ y\in (-L_c, 0], \\
 \hat{w}(t, 0)=0, & t\in\mathbb{R}, \\
 \hat{w}_y(t,0)=-\frac{c}{\mu}, & t\in\mathbb{R}.
 \end{array}\right.
 \end{align*}
 From relation $v_n(t, z)=w(t, z-H_{c,n}(t))$, we have
 \begin{align}\label{vhatwhat}
 \hat{v}(t,z)=\hat{w}(t, z-L_c)
 \end{align}
 for $0<z<L_c$. Since $\hat{v}(t,0)=0$, we see that $\lim_{y\to -L_c}\hat{w}(t, y)=\lim_{y\to -L_c}\hat{v}(t, y+L_c)=0$. So we have $\hat{w}\in C^{1,2}(\mathbb{R}\times [-L_c, 0])$ and 
 \begin{align*}
 \left\{
  \begin{array}{ll}
 \hat{w}_t=\hat{w}_{yy}+c\hat{w}_y+\hat{w}(1-\hat{w}), & t\in\mathbb{R},\ y\in (-L_c, 0], \\
 \hat{w}(t, -L_c)=\hat{w}(t, 0)=0, & t\in\mathbb{R}, \\
 \hat{w}_y(t,0)=-\frac{c}{\mu}, & t\in\mathbb{R}.
 \end{array}\right.
 \end{align*}
 
 We use zero number argument from \cite{DLZ} and \cite{GLZ} to conclude that $\hat{w}(t,y)\equiv V_c(y+L_c)$. Suppose that $\hat{w}(t, y)\not\equiv V_c(y+L_c)$. Then there exist 
 $t_0\in\mathbb{R}$ and $y_0\in (-L_c, 0)$ such that $\hat{w}(t_0, y_0)\ne V_c(y_0+L_c)$. By continuity, we see that there exists $\varepsilon>0$ such that 
 $\hat{w}(t, y_0)\ne V_c(y_0+L_c)$ for $t\in [t_0-\varepsilon, t_0+\varepsilon]$. Now we consider $\eta(t,y):=\hat{w}(t, y)-V_c(y+L_c)$ over $[t_0-\varepsilon, t_0+\varepsilon]\times [y_0,0]$. It is 
 clear that $\eta$ satisfies
 \begin{align*}
 \left\{
 \begin{array}{ll}
 \eta_t=\eta_{yy}+\tilde{m}(t,y)\eta, &(t,y)\in[t_0-\varepsilon, t_0+\varepsilon]\times [y_0, 0], \\
 \eta(t, y_0)\ne 0,\ \eta(t,0)=0, & t\in [t_0-\varepsilon, t_0+\varepsilon],
 \end{array}\right.
 \end{align*}
 where $\tilde{m}$ is a bounded function. 
 
 Therefore we can use the result of zero number by \cite{A}(see Lemma \ref{zeronumber0}) to conclude that the number of zeros of $\eta(t,\cdot)$ on $[y_0, 0]$, say $\mathcal{Z}(t)$, is finite and nonincreasing. 
 Furthermore, if $\eta(s_0, \cdot)$ has a degenerate zero on $[y_0, 0]$ for some $s_0\in (t_0-\varepsilon, t_0+\varepsilon)$, then for any $t_0-\varepsilon<t<s_0<s<t_0+\varepsilon$ we have
 \begin{align*}
 \mathcal{Z}(s)<\mathcal{Z}(s_0)<\mathcal{Z}(t).
 \end{align*}
 However, since 
 \begin{align*}
 &\eta(t, 0)=\hat{w}(t,0)-V_c(L_c)=0, \\
 &\eta_y(t,0)=\hat{w}_y(t,0)-V_c'(L_c)=0,
 \end{align*}
 $\eta(t,\cdot)$ has degenerate zero $y=0$ for any $t\in [t_0-\varepsilon, t_0+\varepsilon]$. Since $\mathcal{Z}(t)<\infty$, this is a contradiction. Thus, we have shown that $\hat{w}(t, y)\equiv V_c(y+L_c)$. From 
 \eqref{vhatwhat}, we also have $\hat{v}(t,z)\equiv V_c(z)$ on $\mathbb{R}\times[0,H_c)$.

Since $(L_c, V_c)$ is uniquely determined by \eqref{compact-support-wave} and Proposition \ref{compact-support-wave-ex}, and thus does not depend on any subsequence of $\{t_n\}$, we can conclude that
\begin{align}
&\lim_{t\to\infty}\left\{\sup_{y\in [-L,0]}|w(t,y)-V_c(y+L_c)|\right\}=0, \label{wconvtoVc} \\
&\lim_{t\to\infty}\left\{\sup_{z\in [0, L]}|v(t,z)-V_c(z)|\right\}=0  \label{vconvtoVc}
\end{align}
holds for any $L\in (0,L_c)$. From \eqref{wconvtoVc} and \eqref{vconvtoVc}, we obtain \eqref{transitionconv}.
\end{proof}
\section{Proof of Theorem B}

In this section, we prove Theorem B. Although we follow the proof in section 4 of \cite{DWZ}, we have to notice 
that the maximal existence time of the solution might be finite. So we divide the proof into several lemmas. In particular, we need Lemma \ref{54} below for our model.

Fix $\phi\in\mathscr{X}(h_0)$, and for $\sigma>0$ let $(u_{\sigma}, h_{\sigma})$ denote the unique positive solution of \eqref{fbp} with initial function $u_0=\sigma\phi$. We assume that $(u_{\sigma}, h_{\sigma})$ 
is defined for $t\in(0, T_{\sigma}^*)$ with $T_{\sigma}^*$ denoting its maximal existence time. Following \cite{DWZ}, we call ``$(u_{\sigma}, h_{\sigma})$ is vanishing (spreading, transition)'', if 
case (i) ((ii), (iii), respectively) in Theorem A happens for $(u_{\sigma}, h_{\sigma})$.

The following lemma follows from the comparison principle.

\begin{lemma}\label{51}
\begin{enumerate}[{\rm (1)}]
\item If $\sigma_1\le \sigma_2$, then $T_{\sigma_1}^*\le T_{\sigma_2}^*$.
\item If $(u_{\sigma_1}, h_{\sigma_1})$ is vanishing, then for $\sigma\in(0,\sigma_1)$, $(u_{\sigma}, h_{\sigma})$ is also vanishing.
\item If $(u_{\sigma_2}, h_{\sigma_2})$ is spreading, then for $\sigma\in(\sigma_2, \infty)$, $(u_{\sigma}, h_{\sigma})$ is also spreading.
\end{enumerate}
\end{lemma}

Define 
\begin{align*}
\Sigma_1:=\{\sigma>0 : (u_{\sigma}, h_{\sigma})\ \mbox{is\ vanishing}\}, \Sigma_2:=\{\sigma>0 : (u_{\sigma}, h_{\sigma})\ \mbox{is\ spreading}\}
\end{align*}
and
\begin{align*}
\sigma_*:=\sup\Sigma_1, \sigma^*=\inf\Sigma_2.
\end{align*}
\begin{lemma}\label{52}
\begin{enumerate}[{\rm (1)}]
\item $0<\sigma_*\le\sigma^*$.
\item If $L_c<h_0$, then $\sigma^*<\infty$.
\end{enumerate}
\end{lemma}
\begin{proof}
\begin{enumerate}[(1)]
\item The fact $\sigma_*>0$ follows from Corollary \ref{suffvanish2}. The fact $\sigma_*\le\sigma^*$ follows from their definition and Lemma \ref{51}.
\item If $L_c<h_0$, then we can find $\sigma>0$ large enough such that
\begin{align*}
\sigma\phi(x)\ge V_c(x)\ \ {\rm for}\ x\in[0, L_c].
\end{align*}
Therefore by Lemma \ref{suffforspreading}, $(u_{\sigma}, h_{\sigma})$ is spreading. Thus we have $\sigma^*<\infty$.
\end{enumerate}
\end{proof}
\begin{lemma}\label{53} $\sigma_*\notin\Sigma_1$.
\end{lemma}
\begin{proof}
Suppose that $\sigma_*\in\Sigma_1$. We have $T_{\sigma_*}^*<\infty$ and
\begin{align*}
\lim_{t\nearrow T_{\sigma_*}^*}(h_{\sigma_*}(t)-ct)=0, \lim_{t\nearrow T_{\sigma_*}^*}\|u_{\sigma_*}(t, \cdot)\|_{C[ct, h(t)]}=0.
\end{align*}
Then, by Lemma \ref{fte1}, we can find $T_0\in (0, T_{\sigma_*}^*)$ such that
\begin{align*}
0<h_{\sigma_*}(T_0)-cT_0<\frac{h_0}{2},\ \|u_{\sigma_*}(T_0, \cdot)\|_{C[cT_0, h_{\sigma_*}(T_0)]}\le\frac{1}{2}C(h_0, c, \mu),
\end{align*}
where $C(h_0, c, \mu)$ is the constant defined in Corollary \ref{suffvanish2}. By the continuous dependence of solution $(u_{\sigma}, h_{\sigma})$ on $\sigma$, we can see 
that for some sufficiently small $\varepsilon>0$
\begin{align*}
0<h_{\sigma_*+\varepsilon}(T_0)-cT_0<h_0,\ \|u_{\sigma_*+\varepsilon}(T_0, \cdot)\|_{C[cT_0, h_{\sigma_*+\varepsilon}(T_0)]}<C(h_0, c, \mu)
\end{align*}
holds. Therefore, by Corollary \ref{suffvanish2}, we can conclude that $(u_{\sigma_*+\varepsilon}, h_{\sigma_*+\varepsilon})$ is vanishing, contradicting the definition of $\sigma_*$. 
The proof is completed.
\end{proof}
\begin{lemma}\label{54} Assume that $T^*_{\sigma_0}=\infty$ for some $\sigma_0>0$. Then we have $\sup_{\sigma\in(0,\sigma_0)}T_{\sigma}^*=\infty$.
\end{lemma}
\begin{proof}
Suppose that $\sup_{\sigma\in (0, \sigma_0)}T_{\sigma}^*<\infty$. Since $T_{\sigma}^*$ is nondecreasing in $\sigma$, we have
\begin{align*}
\overline{T}:=\sup_{\sigma\in (0, \sigma_0)}T_{\sigma}^*=\lim_{\sigma\nearrow\sigma_0}T_{\sigma}^*.
\end{align*}
By the assumption of the lemma, we can see that
\begin{align*}
\rho:=\inf_{t\in[0,\overline{T}]}(h_{\sigma_0}(t)-ct)>0
\end{align*}
and for any $T\in (0, \overline{T})$ we have
\begin{align*}
\inf_{t\in [0,T]}(h_{\sigma_0}(t)-ct)>\rho>0.
\end{align*}
From continuous dependence of solutions on $\sigma$, we have that for any $T\in(0, \overline{T})$ there exists $\tilde{\sigma}>0$ such that
\begin{align*}
\inf_{t\in [0, T]}(h_{\sigma}(t)-ct)>\rho
\end{align*}
for $\sigma\in(\tilde{\sigma}, \sigma_0)$.  By Lemma \ref{hprime} we also have
\begin{align*}
&0<u_{\sigma}(t,x)\le C_1\ {\rm for}\ t\in[0, T_{\sigma}^*)\ {\rm and}\ x\in (ct, h_{\sigma}(t)), \\
&0<h'_{\sigma}(t)\le \mu C_2\ {\rm for}\ t\in [0, T_{\sigma}^*)
\end{align*}
for any $\sigma\in (0, \sigma_0)$ and 
\begin{align*}
&0<u_{\sigma_0}(t,x)\le C_1\ {\rm for}\ t\in[0, \overline{T}]\ {\rm and}\ x\in (ct, h_{\sigma}(t)), \\
&0<h'_{\sigma_0}(t)\le \mu C_2\ {\rm for}\ t\in [0, \overline{T}],
\end{align*}
where $C_1$, $C_2$ are constants which depend on $\phi$ and $\sigma_0$ but not on $\sigma\in(0, \sigma_0)$.

We now fix $\delta\in (0, \overline{T})$. By the standard $L^p$ estimates, the Sobolev embedding theorem and the Schauder estimates for parabolic 
equations, we can find $C_3>0$ depending only on $\delta$, $\overline{T}$, $C_1$, $C_2$ such that
\begin{align*}
\|u_{\sigma_0}(t, \cdot)\|_{C^2[ct, h_{\sigma_0}(t)]}\le C_3\ \ {\rm for}\ \ t\in(\delta, \overline{T}].
\end{align*}
Now fix $t\in(\delta, \overline{T})$ arbitrarily. By continuous dependence of solutions on $\sigma$, we can find $\sigma\in (0,\sigma_0)$ such that 
\begin{align}\label{c2normest}
\|u_{\sigma}(t, \cdot)\|_{C^2[ct, h_{\sigma}(t)]}\le 2C_3.
\end{align}
It follows from the proof of Proposition \ref{existence}(see \cite{DLI}) that there exists $\tau>0$ depending only on $C_1$, $C_2$, $C_3$ and $\rho$ but not on $t$ such that 
solution $(u_{\sigma}, h_{\sigma})$ with initial time $t$ can be extended uniquely to the time $t+\tau$. If we choose $t=\overline{T}-\frac{\tau}{2}$, then  we can find $\sigma\in(0,\sigma_0)$ such 
that $(u_{\sigma}, h_{\sigma})$ satisfies \eqref{c2normest} with $t=\overline{T}-\frac{\tau}{2}$, and we can extend the solution up to  $\overline{T}+\frac{\tau}{2}$. This is a contradiction to 
the definition of $\overline{T}$.
\end{proof}

\begin{lemma}\label{55}
$\sigma^*\notin\Sigma_2$
\end{lemma}
\begin{proof}
Suppose that $\sigma^*\in\Sigma_2$. Since $(u_{\sigma^*}, h_{\sigma^*})$ is spreading, 
\begin{align}\label{spreadingsigmastar1}
\lim_{t\to\infty}(h_{\sigma^*}(t)-ct)=\infty
\end{align}
holds. Moreover, by Proposition \ref{spreadingconv}, for any $\varepsilon>0$, we have
\begin{align}\label{spreadingsigmastar2}
\lim_{t\to\infty}\max_{x\in [(c+\varepsilon)t, (c^*-\varepsilon)t]}|u_{\sigma^*}(t,x)-1|=0.
\end{align}
Fix any small $\varepsilon>0$. From \eqref{spreadingsigmastar1} and \eqref{spreadingsigmastar2}, we can find $T_0>0$ such that
\begin{align*}
&(c^*-\varepsilon)T_0-(c+\varepsilon)T_0>L_c, \\
&\min_{x\in[(c^*-\varepsilon)T_0, (c+\varepsilon)T_0]}u_{\sigma^*}(T_0, x)>\max_{x\in[0, L_c]}V_c(x).
\end{align*}
From Lemma \ref{54} and Lemma \ref{51}, we can choose $\sigma_0\in (0, \sigma^*)$ such that $T_0<T_{\sigma}^*$ holds for $\sigma\in (\sigma_0, \sigma^*)$.  
By continuous dependence of solutions on $\sigma$, we can find $\sigma\in(\sigma_0, \sigma^*)$ close to $\sigma^*$ such that
\begin{align*}
\min_{x\in[(c+\varepsilon)T_0, (c^*-\varepsilon)T_0]}u_{\sigma}(T_0,x)>\max_{x\in[0, L_c]}V_c(x)
\end{align*}
and then
\begin{align*}
u_{\sigma}(T_0, x)>V_c(x-(c+\varepsilon)T_0).
\end{align*}
From Lemma \ref{suffforspreading}, we can conclude that $\lim_{t\to\infty}(h_{\sigma}(t)-ct)=\infty$, contradicting the definition of $\sigma^*$. Therefore, we have shown that 
$\sigma^*\in\Sigma_2$. 
\end{proof}
\begin{lemma}\label{56}
$\sigma_*=\sigma^*$
\end{lemma}
\begin{proof}
Suppose that $\sigma_*<\sigma^*$ and consider $(u_{\sigma_*}, h_{\sigma_*})$ and $(u_{\sigma^*}, u_{\sigma^*})$. If $\sigma^*=\infty$ we can choose any $\sigma\in (\sigma_*, \infty)$ and consider $(u_{\sigma}, h_{\sigma})$ in stead of 
$(u_{\sigma^*}, h_{\sigma^*})$. By definitions of $\Sigma_1$ and $\Sigma_2$, both $(u_{\sigma_*}, h_{\sigma_*})$ and $(u_{\sigma^*}, h_{\sigma^*})$ are transition. Thus we have
\begin{align}\label{560} 
\lim_{t\to\infty}(h_{\sigma_*}(t)-ct)=\lim_{t\to\infty}(h_{\sigma^*}(t)-ct)=L_c.
\end{align}
By the comparison principle and the strong maximum principle, we have
\begin{align}\label{561}
h_{\sigma_*}(t)<h_{\sigma^*}(t)\ \ {\rm for}\ \ t>0
\end{align}
and
\begin{align}\label{562}
u_{\sigma_*}(t,x)<u_{\sigma^*}(t, x)\ \ {\rm for}\ \ t>0,\ \ ct\le x\le h_{\sigma_*}(t).
\end{align}
Fix $t_0>0$. Then by \eqref{561} and \eqref{562}, we can choose $\tau_0>0$ small such that
\begin{align*}
&h_{\sigma_*}(t_0)+\tau_0\le h_{\sigma^*}(t_0), \\
&u_{\sigma_*}(t_0, x-\tau_0)\le u_{\sigma^*}(t_0,x)\ \ {\rm for}\ \ ct_0+\tau_0\le x\le h_{\sigma_*}(t_0)+\tau_0. 
\end{align*}
Moreover by \eqref{560}, we may assume that $h_{\sigma^*}(t)>ct+\tau_0$ for all $t>t_0$.

Now we define 
\begin{align*}
&\underline{h}(t):=h_{\sigma_*}(t)+\tau_0, \\
&\underline{u}(t, x):=u_{\sigma_*}(t, x-\tau_0), \\
&\xi(t):=ct+\tau_0.
\end{align*}
Then, we have
\begin{align*}
&\underline{u}_t=\underline{u}_{xx}+\underline{u}(1-\underline{u}),\ \ {\rm for}\ \ t>t_0,\ \ \xi(t)<x<\underline{h}(t), \\
&0=\underline{u}(t, \xi(t))\le u_{\sigma^*}(t, \xi(t)), \ \ {\rm for}\ \ t>t_0, \\
&\underline{u}(t, \underline{h}(t))=u_{\sigma_*}(t, h_{\sigma_*}(t))=0,\ \ {\rm for}\ \ t>t_0, \\
&\underline{u}(t_0, x)=u_{\sigma_*}(t_0, x-\tau_0)\le u_{\sigma^*}(t_0, x),\ \ {\rm for}\ \ \xi(0)\le x\le \underline{h}(0), \\
&\underline{h}'(t)=-\mu\underline{u}_x(t, \underline{h}(t)),\ \ {\rm for}\ \ t>t_0.
\end{align*}
The comparison principle (Lemma \ref{comp}) implies that $\underline{h}(t)\le h_{\sigma^*}(t)$, that is, $h_{\sigma_*}(t)+\tau_0\le h_{\sigma^*}(t)$ for $t>t_0$. However, by \eqref{560} we have
\begin{align*}
L_c+\tau_0=\lim_{t\to\infty}(h_{\sigma_*}(t)-ct)+\tau_0\le\lim_{t\to\infty}(h_{\sigma^*}(t)-ct)=L_c,
\end{align*}
which is a contradiction. Therefore we have completed the proof of the Lemma.
\end{proof}
\begin{proof}[Proof of Theorem B]
Lemmas \ref{51} -- \ref{56} lead to the statement of Theorem B.
\end{proof}
\section{Proof of Theorem C}
In this section, we consider the case where $c\ge c^*$. Let $(u, h)$ be the unique solution to \eqref{fbp} defined for $t\in(0, T^*)$ with $T^*$ maximal existence time of the solution.

\begin{lemma}
If $c>c^*$, then $(u, h)$ is always vanishing, that is, $T^*<\infty$ and 
\begin{align*}
\lim_{t\nearrow T^*}(h(t)-ct)=0,\ \lim_{t\nearrow T^*}\sup_{x\in [ct, h(t)]}|u(t,x)|=0.
\end{align*}
\end{lemma}
\begin{proof}
From Proposition \ref{upperestspeed}, we have $h(t)\le c^*t+C_0$ for some constant $C_0>0$, which yields that $T^*<\infty$.
\end{proof}

\begin{proposition}
If $c=c^*$, then $(u,h)$ is always vanishing.
\end{proposition}
\begin{proof}
Arguing indirectly we assume that $T^*=\infty$. 

{\bf Step 1.} $\lim_{t\to\infty}H_{c^*}(t)$ exists.

To show this claim, it suffices to show that for any $b\in(0,\infty)$, $H_{c^*}(t)-b$ changes its sign at most finitely many times.

Define $v(t,z)=u(t, z+c^*t)$ and $\eta(t,z)=v(t,z)-q^*(z-b)$, where $q^*$ is defined in Proposition \ref{semi-wave}. 
Then $\eta(t,z)$ satisfies a linear parabolic equation over $\{(t, z)\in\mathbb{R}^2\ |\ t>0,\ z\in(0, \min\{H_{c^*}(t), b\})\}$. By a similar zero number argument to the proof of Lemma 4.5 or Lemma 3.7 in \cite{KM}, we can 
show that $H_{c^*}(t)-b$ changes its sign at most finitely many times.

{\bf Step 2.} Reaching a contradiction.

By step 1, $H^*_{c^*}:=\lim_{t\to\infty}H_{c^*}(t)$ exists. Take any sequence $\{t_n\}\subset\mathbb{R}$ with $\lim_{n\to\infty}t_n=\infty$ and define $w(t,y):=v(t, y+H_{c^*}(t))$, 
\begin{align*}
H_{c^*,n}(t):=H_{c^*}(t+t_n),\ v_n(t, z):=v(t+t_n, z)\ \ {\rm and}\ \ w_n(t, y):=w(t+t_n, y).
\end{align*}
By the same arguments as in the proof of Proposition \ref{HctoLc}, for any $\alpha\in(0,1)$, there exists a subsequence of $\{t_n\}$ such that
\begin{align*}
&H_{c^*,n}'(t)\to 0\ \ C^{\frac{\alpha}{2}}_{\rm loc}(\mathbb{R}),\\
&w_n\to\hat{w}\ \ {\rm in}\ \ C^{1+\frac{\alpha}{2}, 2+\alpha}_{\rm loc}(\mathbb{R}\times (-H^*_{c^*}, 0]), \\
&v_n\to\hat{v}\ \ {\rm in}\ \ C^{1+\frac{\alpha}{2}, 2+\alpha}_{\rm loc}(\Omega_0),
\end{align*}
along the subsequence, where $\Omega_0=\{(t,z) : t\in\mathbb{R},\ z\in [0, H^*_{c^*})\}$. Moreover $\hat{w}$ and $\hat{v}$ satisfies
\begin{align*}
\left\{
\begin{array}{ll}
\hat{w}_t=\hat{w}_{yy}+c^*\hat{w}_y+\hat{w}(1-\hat{w}), & t\in\mathbb{R}, -H^*_{c^*}<z<0. \\
\hat{w}(t, 0)=0,& t\in\mathbb{R}, \\ 
\hat{w}_y(t,0)=-\frac{c^*}{\mu}, & t\in\mathbb{R},
\end{array}\right.
\end{align*}
and
\begin{align*}
\hat{v}_t=\hat{v}_{zz}+c^*\hat{v}_z+\hat{v}(1-\hat{v}),\ \ t\in\mathbb{R},\ 0<z<H^*_{c^*}.
\end{align*}
By relation $v_n(t,y+H_{n, c^*}(t))=w_n(t, y)$, we have $\hat{v}(t, y+H^*_{c^*})=\hat{w}(t, y)$ for $t\in\mathbb{R}$ and $z\in (-H^*_{c^*}, 0)$ and 
\begin{align*}
\lim_{y\to -H^*_{c^*}}\hat{w}(t, y)=\lim_{y\to -H^*_{c^*}}\hat{v}(t, y+H^*_{c^*})=0.
\end{align*}
Thus $\hat{w}\in C^{1,2}(\mathbb{R}\times [-H^*_{c^*}, 0])$ and 
\begin{align*}
\left\{
\begin{array}{ll}
\hat{w}_t=\hat{w}_{yy}+c^*\hat{w}_y+\hat{w}(1-\hat{w}), & t\in\mathbb{R}, -H^*_{c^*}<z<0. \\
\hat{w}(t, -H^*_{c^*})=\hat{w}(t,0)=0,& t\in\mathbb{R}, \\ 
\hat{w}_y(t,0)=-\frac{c^*}{\mu}, & t\in\mathbb{R}.
\end{array}\right.
\end{align*}
We now consider $\tilde{\eta}(t, y):=\hat{w}(t, y)-q^*(y)$ for $t\in\mathbb{R}$ and $y\in [-H^*_{c^*}, 0]$. Since $\tilde{\eta}(t, -H^*_{c^*})<0$ and $\tilde{\eta}(t,0)=0$ for any $t\in\mathbb{R}$, we can use the 
result of zero number by \cite{A}(see also Lemma \ref{zeronumber0}) to conclude that the zero number of $\tilde{\eta}(t, \cdot)$ on $(-H^*_{c^*}, 0]$, say $\tilde{\mathcal{Z}}(t)$, is finite and nonincreasing. Furthermore if 
$\tilde{\eta}(s_0, \cdot)$ has a degenerate zero on $(-H^*_{c^*}, 0]$ for some $s_0$, then we have $\mathcal{\tilde{Z}}(t)\le \mathcal{\tilde{Z}}(s)-1$ for any $t<s_0<s$. 
On the other hand $\tilde{\eta}(t, \cdot)$ has degenerate zero $y=0$ for any $t\in\mathbb{R}$, which is contradiction. The proof is complete.
\end{proof}

\section*{Appendix}
\renewcommand{\theequation}{A.\arabic{equation}}
\setcounter{equation}{0}

In this appendix, we give the proof of Proposition \ref{spreadingconv}. Although the proof is almost identical to that of Theorem 3.9 in \cite{DWZ}, we give its sketch for reader's convenience. 

\begin{proof}[Proof of Proposition \ref{spreadingconv}]
Assume that the statement of the proposition does not hold. Then there exist $\theta_0>0$, $\varepsilon_0>0$ and a sequence of points $(t_n, x_n)\in\mathbb{R}^2$ with 
$t_n\to\infty$ as $n\to\infty$ and $x_n\in [(c+\varepsilon_0)t_n, (c^*-\varepsilon_0)t_n]$ such that
\begin{align*}
|u(t_n, x_n)-1|>\theta_0.
\end{align*}
Choose $\delta>0$ small so that
\begin{align*}
c^*-\frac{\varepsilon_0}{2}\le (1-\delta)\left(c^*-\frac{\varepsilon_0}{3}\right).
\end{align*}
Now we define
\begin{align*}
&\gamma_n:=(1-\delta)t_n, \\
&\Omega_n:=\left\{(t,x)\ :\ 0\le t\le t_n-\gamma_n, -\frac{\varepsilon_0}{2}t_n<x<\frac{\varepsilon_0}{2}t_n\right\}
\end{align*}
and 
\begin{align*}
v_n(t, x):=u(t+\gamma_n, x+x_n)\ \ {\rm for}\ \ (t,x)\in\Omega_n.
\end{align*}
Then, in view of our choice of $\delta$ and the fact $\lim_{t\to\infty}(h(t)/t)=c^*$, we have for all large $n$
\begin{align}\label{4901}
\begin{split}
x+x_n\le\left(c^*-\frac{\varepsilon_0}{2}\right)t_n&\le (1-\delta)\left(c^*-\frac{\varepsilon_0}{3}\right)t_n \\            
                                               &\le\left(c^*-\frac{\varepsilon_0}{3}\right)\gamma_n (\le h(\gamma_n)\le h(t+\gamma_n))
\end{split}
\end{align}
and 
\begin{align}\label{4902}
\begin{split}
x+x_n\ge\left(c+\frac{\varepsilon_0}{2}\right)t_n&>\left(c+\frac{\varepsilon_0}{3}\right)t_n \\ 
                                                 &>\left(c+\frac{\varepsilon_0}{3}\right)(t+\gamma_n)(>c(t+\gamma_n)) 
\end{split}
\end{align}
for all $(t, x)\in\Omega_n$. Thus $v_n$ is well defined. Moreover, we have
\begin{align}\label{4903}
|v_n(t_n-\gamma_n, 0)-1|\ge \theta_0\ \ {\rm for\ all}\ \ n
\end{align}
and from the first inequality of \eqref{4902}, we have for all $(t,x)\in\Omega_n$
\begin{align*}
x+x_n-c(t+\gamma_n)\ge\frac{\varepsilon_0}{2}t_n\to\infty
\end{align*}
as $n\to\infty$. 

Clearly, $v_n$ satisfies
\begin{align*}
\frac{\partial v_n}{\partial t}=\frac{\partial^2 v_n}{\partial x^2}+v_n(1-v_n)\ \ {\rm for}\ \ (t,x)\in\Omega_n.
\end{align*}
By Lemma \ref{ODE-estimate}
\begin{align}\label{4904}
\varlimsup_{t\to\infty}\sup_{x\in[0, h(t)]}u(t,x)\le 1.
\end{align}
Combining \eqref{4903} and \eqref{4904}, we obtain
\begin{align}\label{4905}
v_n(t_n-\gamma_n, 0)\le 1-\theta_0\ \ 
\end{align}

Now we consider the unique positive solution $\tilde{w}_l^*$ of the problem
\begin{align*}
\left\{
\begin{array}{ll}
-w''=w(1-w),& -l<x<l, \\
w(-l)=w(l)=0&
\end{array}\right.
\end{align*}
for all large $l$. We can choose  large $l$ such that $\tilde{w}_l^*(0)\ge 1-\theta_0$ (Lemma \ref{auxelliptic}, see also \cite{D}). Fix such an $l$. For all large $n$ we have $\frac{\varepsilon_0}{2}t_n>l$ and
\begin{align*}
\left\{
\begin{array}{ll}
\displaystyle\frac{\partial v_n}{\partial t}=\frac{\partial^2 v_n}{\partial x^2}+v_n(1-v_n), & 0<t<t_n-\gamma_n,\ -l<x<l, \\
v_n(t, -l)>0,\ v_n(t,l)>0, & 0<t<t_n-\gamma_n, \\ 
v_n(0,x)=u(\gamma_n, x+x_n), & -l<x<l.
\end{array}\right.
\end{align*}

We will show that there exists $\beta>0$ such that
\begin{align}\label{4907}
u(\gamma_n, x+x_n)\ge \beta\ \ {\rm for}\ \ -l<x<l
\end{align}
for all large. If we assume \eqref{4907}, we can derive a contradiction. In fact, let $\tilde{w}_l(t,x)$ be the unique positive solution of
\begin{align*}
\left\{
\begin{array}{ll}
w_t=w_{xx}+w(1-w),& t>0,\ -l<x<l,\\
w(t, -l)=w(t,l)=0, & t>0, \\ 
w(0,x)=\beta, & -l<x<l.
\end{array}\right.
\end{align*}
Since $\tilde{w}_l(t,\cdot)$ converges to $\tilde{w}_l^*$ uniformly as $t\to\infty$, there exists $T^{\star}>0$ such that
\begin{align}\label{4909}
\tilde{w}_l(t,0)>1-\theta_0\ \ {\rm for}\ \ t>T^{\star}.
\end{align}
By virtue of \eqref{4907}, we can use the standard comparison principle to obtain
\begin{align}\label{4910}
\tilde{w}_l(t, x)\le v_n(t,x)\ \ {\rm for}\ \ 0\le t\le t_n-\gamma_n,\ \ -l<x<l.
\end{align}
Since $t_n-\gamma_n=\delta t_n\to\infty$ as $n\to\infty$, by \eqref{4909} and \eqref{4910}, we can conclude that
\begin{align*}
v_n(t_n-\gamma_n, 0)\ge \tilde{w}_l(t_n-\gamma_n, 0)>1-\theta_0
\end{align*}
for all large $n$ satisfying $t_n-\gamma_n>T^{\star}$, which contradicts \eqref{4905}.

To complete the proof, we have to show \eqref{4907}. As in the proof of Theorem 3.9 in \cite{DWZ}, we will look more closely the proof of Lemma \ref{Hcunbounded} and see that we can obtain 
the estimate for $u$ uniformly in $\tilde{c}$ in an interval.

We first observe that
\begin{align*}
[-l,l]\subset\left[-\frac{\varepsilon_0}{2}t_n, \frac{\varepsilon_0}{2}t_n\right]\ \ \ {\rm for\ all\ large}\ n,
\end{align*}
and hence by \eqref{4901} and \eqref{4902}, we have
\begin{align}\label{4911}
\left(c+\frac{\varepsilon_0}{3}\right)\gamma_n\le x+x_n\le\left(c^*-\frac{\varepsilon}{3}\right)\gamma_n\ \ {\rm for}\ \ x\in[-l,l]\ \ {\rm and\ all\ large}\ n.
\end{align}
We set $I_0:=\left[c+\frac{\varepsilon_0}{4}, c^*-\frac{\varepsilon_0}{4}\right]$. We second observe that there exists $\varepsilon>0$ such that
\begin{align*}
\sup_{c\in I_0}\|V_c\|_{\infty}<1-\varepsilon.
\end{align*}
We also observe that
\begin{align*}
\overline{L}:=\sup_{c\in I_0}L_c<\infty.
\end{align*}
Thus we can choose $\tilde{l}>\sup_{c\in I_0}L_c$ which is independent of $\tilde{c}\in I_0$ such that 
\begin{align*}
w_{\tilde{l}}(x)>1-\frac{\varepsilon}{2}\ \ {\rm for}\ \ -\frac{\overline{L}}{2}<x<\frac{\overline{L}}{2},
\end{align*}
where $w_{\tilde{l}}$ is a unique positive solution of
\begin{align*}
\left\{
\begin{array}{ll}
w''+cw'+w(1-w)=0, & -\tilde{l}<x<\tilde{l}, \\
w(-\tilde{l})=w(\tilde{l})=0. &
\end{array}\right.
\end{align*}
Following the proof of Lemma \ref{Hcunbounded}, we see that we can choose $\psi_0$, $T=T(\psi_0)>0$ and $T_1>0$ independently of $\tilde{c}\in I_0$. Hence 
we see from \eqref{psi2} that 
\begin{align}\label{4912}
v(T_1+T, z)>1-\varepsilon\ \ {\rm for}\ \ \tilde{l}-\frac{\overline{L}}{2}<z<\tilde{l}+\frac{\overline{L}}{2}.
\end{align} 
Denote $T_0=T+T_1$. We note that $T_0$ is independent of $\tilde{c}\in I_0$. \eqref{4912} implies that 
\begin{align*}
u(T_0, x)>1-\varepsilon\ \ {\rm for}\ \ x\in\left[\tilde{l}+cT_0-\frac{L_{\tilde{c}}}{2}, \tilde{l}+cT_0+\frac{L_{\tilde{c}}}{2}\right]\subset \left[\tilde{l}+cT_0-\frac{\overline{L}}{2}, \tilde{l}+cT_0+\frac{\overline{L}}{2}\right].
\end{align*}
and then
\begin{align*}
u(T_0, x)>V_{\tilde{c}}\left(x-\tilde{l}-cT_0+\frac{L_{\tilde{c}}}{2}\right)\ \ {\rm for}\ \ x\in \left[\tilde{l}+cT_0-\frac{L_{\tilde{c}}}{2}, \tilde{l}+cT_0+\frac{L_{\tilde{c}}}{2}\right]
\end{align*}
for any $\tilde{c}\in I_0$. By the same argument of the proof of Lemma \ref{Hcunbounded}, we obtain
\begin{align*}
u(t,x)\ge V_{\tilde{c}}\left(x-\tilde{c}t-\tilde{l}-cT_0+\frac{L_{\tilde{c}}}{2}\right).
\end{align*}
Taking $x=\tilde{c}t+\tilde{l}+cT_0$, we obtain
\begin{align*}
u(t+T_0, \tilde{c}t+\tilde{l}+cT_0)\ge V_{\tilde{c}}\left(\frac{L_{\tilde{c}}}{2}\right)\ \ \ {\rm for}\ \ t>0,\ \ \tilde{c}\in I_0.
\end{align*}
We can find that there exists $\beta>0$ such that 
\begin{align*}
V_{\tilde{c}}\left(\frac{L_{\tilde{c}}}{2}\right)\ge \beta.
\end{align*}
Hence, if we take $t+T_0=\gamma_n$, then
\begin{align*}
\tilde{c}t+\tilde{l}+cT_0=\tilde{c}\gamma_n+(c-\tilde{c})T_0+\tilde{l}
\end{align*}
and
\begin{align}\label{4915}
u(\gamma_n, \tilde{c}\gamma_n+(c-\tilde{c})T_0+\tilde{l})\ge\beta
\end{align}
for all large $n$ and all $\tilde{c}\in I_0$. By definition of $I_0$, we have for all large $n$
\begin{align*}
\{\tilde{c}\gamma_n+(c-\tilde{c})T_0+\tilde{l}:\tilde{c}\in I_0\}\supset\left[\left(c+\frac{\varepsilon_0}{3}\right)\gamma_n, \left(c^*-\frac{\varepsilon_0}{3}\right)\gamma_n\right].
\end{align*}
Therefore from \eqref{4911} and \eqref{4915} we have obtained \eqref{4907}. The proof have been completed.
\end{proof}

\section*{Acknowledgement}

The author would like to thank Professor Kazuhiro Ishige, Tohoku University, Japan, for his valuable suggestion which leads to this project.

\ \\



\begin{thebibliography}{99}
\bibitem{A} S. B. Angenent, {\it The zero set of a solution of a parabolic equation}, J. Reine Angew. Math., 390 (1988), 79--96.
\bibitem{AW1} D. G. Aronson and H. F. Weinberger, {\it Nonlinear diffusion in population genetics, combustion, and nerve pulse propagation}, in Partial Differential Equations and Related Topics, 
Lecture Notes in Math. 446, Springer, Berlin, 1975, 5--49.
\bibitem{AW2} D. G. Aronson and H. F. Weinberger, {\it Multidimensional nonlinear diffusion arising in population genetics}, Adv. in Math., 30 (1978), 33--76.
\bibitem{C} J. Cai, {\it Asymptotic behavior of solutions of Fisher-KPP equation with free boundary conditions}, Nonlinear Anal., 16 (2014), 170--177.
\bibitem{CLZ} J. Cai, B. Lou and M. Zhou, {\it Asymptotic behavior of solutions of a reaction diffusion equation with free boundary conditions}, J. Dynam. Differential Equations, 26(2014), 1007--1028.
\bibitem{D} Y. Du, {\it Order Structure and Topological Methods in Nonlinear Partial Differential Equations,\ Vol. 1\ Maximum Principle and Applications,} World Scientific Publishing, 2006.
\bibitem{DLI} Y. Du and Z. Lin, {\it Spreading-vanishing dichotomy in the diffusive logistic model with a free boundary}, SIAM J. Math. Anal., 42 (2010), 377--405.
\bibitem{DL} Y. Du and B. Lou, {\it Spreading and vanishing in nonlinear diffusion problems with free boundaries}, J. Eur. Math. Soc., 17 (2015) 2673--2724.
\bibitem{DLZ} Y. Du, B. Lou and M. Zhou, {\it Nonlinear diffusion problems with free boundaries : Convergence, transition speed and zero number arguments}, SIAM J. Math. Anal., 47(2015), 3555-3584.
\bibitem{DMZ} Y. Du, H. Matsuzawa and M. Zhou, {\it Sharp estimate of the spreading speed determined by nonlinear free boundary problems}, SIAM J. Math. Anal., 46 (2014), 375--396.
\bibitem{DWZ} Y. Du, L. Wei and L. Zhou, {\it Spreading in a shifting environment modeled by the diffusive logistic equation with a free boundary}, preprint. (arXiv:1508.06246)
\bibitem{F} F. J. Fernandez, {\it Unique continuation for parabolic operators. II}, Comm. Partial Differential Equations 28 (2003), 1597--1604. 
\bibitem{GLZ} H. Gu, B. Lou and M. Zhou, {\it Long time behavior of solutions of Fisher-KPP equation with advection and free boundaries}, J. Funct. Anal., 269 (2015) 1714--1768.
\bibitem{KOY} Y. Kaneko, K. Oeda and Y. Yamada, {\it Remarks on spreading and vanishing for free boundary problems of some reaction-diffusion equations}, Funkcial. Ekvac., 57 (2014), 449--465.
\bibitem{KY} Y. Kaneko and Y. Yamada, {\it A free boundary problem for a reaction-diffusion equation appearing in ecology}, Adv. Math. Sci. Appl., 21 (2011), 467--492.
\bibitem{KM} Y. Kaneko and H. Matsuzawa, {\it Spreading speed and sharp asymptotic profiles of solutions in free boundary problems for nonlinear advection-diffusion equations}, J. Math. Anal. Appl., 428 (2015), 43--76. 
\bibitem{KM2} Y. Kaneko and H. Matsuzawa, {\it Spreading and vanishing in a free boundary problem for nonlinear diffusion equations with a given forced moving boundary}, preprint.
\bibitem{LSU} O. A. Ladyzenskaja, V. A. Solonnikov and N. N. Ural' ceva, {\it Linear and Quasilinear Equations of Parabolic Type}, Amer. Math. Soc., Providence, RI, 1968.
\bibitem{Li} G. M. Lieberman, {\it Second Order Parabolic Differential Equations}, World Scientific, Singapore, 1996.
\end{thebibliography}
\end{document}